\documentclass[12pt,leqno]{article}
\usepackage{amsmath,amssymb,amsfonts,amsthm}

 \newcommand{\whp}{with high probability}
 \newcommand{\V}{\mathcal{V}}
\newcommand{\W}{\mathcal{W}}
 \newcommand{\Po}[1]{\textrm{Po}\left(#1\right)}
 \newcommand{\Bin}[2]{\textrm{Bin}\left(#1,#2\right)}
 \newcommand{\E}{\mathbb{E}}
 \newcommand{\Var}{{\rm Var}}
\newcommand{\coupling}{\preceq}
\newcommand{\coup}{\preceq_{1-o(1)}}
\newcommand{\Pra}[1]{\Pr\left\{#1\right\}}
\newcommand{\Gn}[1]{G_2\left(n,#1\right)}
\newcommand{\Gh}[1]{G_3\left(n,#1\right)}
\newcommand{\Gnm}[2]{\mathcal{G}_{#1}\left(n,m,#2\right)}
\newcommand{\Gnmp}{\mathcal{G}\left(n,m,\overline{p}\right)}
\newcommand{\Gnmpp}{\mathcal{G}\left(n,m,p\right)}
\newcommand{\Gnmprim}[1]{\mathcal{G}\left(n,#1,p\right)}
\newcommand{\G}[2]{\mathbb{G}_{*#1}\left(n,#2\right)}
\newcommand{\Gsuma}{\Gn{\hat{p}_2}\cup\Gh{\hat{p}_3}}
\newcommand{\pdwa}{\hat{p}_2}
\newcommand{\ptrzy}{\hat{p}_3}
\newcommand{\phat}{\hat{p}}
\newcommand{\Gdelta}{\rg(n)_{\delta\ge k}}
\newcommand{\rg}{\mathbb{G}}
\theoremstyle{plain}
\newtheorem{thm}{Theorem}

\newtheorem{lem}{Lemma}
\newtheorem{cor}{Corollary}
\newtheorem{fact}{Fact}
\theoremstyle{definition}
\newtheorem{rem}{Remark}

\title{The coupling method for inhomogeneous random intersection graphs.}
\author{Katarzyna Rybarczyk\\
\small Faculty of Mathematics and Computer Science,\\
\small Adam Mickiewicz University, 60--769 Pozna\'n, Poland\\
\small \texttt{kryba@amu.edu.pl}}

\date{}

\begin{document}

\maketitle

\begin{abstract}
We present new results concerning threshold functions for a wide family of random intersection graphs. To this end we improve and generalise  the coupling method introduced for random intersection graphs so that it may be used for a wider range of parameters.  Using the new approach we are able to sharpen the best known results concerning homogeneous random intersection graphs and establish threshold functions for some monotone properties of inhomogeneous random intersection graphs. Considered properties are: $k$-connectivity, matching containment and hamiltonicity.  
\end{abstract}

\section{Introduction}\label{Introduction}

Since their introduction by Karo\'{n}ski, Scheinerman, and Singer--Cohen \cite{GpSubgraph} random intersection graphs have been attracting attention due to their interesting structure and wide applications. The random intersection graph model appears in problems concerning for example ''gate matrix layout'' for VLSI design (see e.g. \cite{GpSubgraph}), cluster analysis and classification (see e.g. \cite{RIGGodehardt1}), analysis of complex networks (see e.g. \cite{RIGClustering2, RIGTunableDegree}), secure wireless networks (see e.g. \cite{WSNphase2}) or epidemics (\cite{GpEpidemics}). Several generalisations of the model has been proposed, mainly in order to adapt it to use for a specific purpose.
In this paper we consider the $\Gnmp$ model studied for example in \cite{Gppdistance,GppDegrees, GppPhaseTransition, SpirakisNiezalezne}. 
Alternative ways of generalizing the model defined in \cite{GpSubgraph} are given for example in \cite{RIGTunableDegree} and \cite{RIGGodehardt1}.   

In a random intersection graph $\Gnmp$ there is  a set of $n$ vertices $\V=\{v_1,\ldots,v_n\}$,   an auxiliary set of $m=m(n)$ features  $\W=\{w_1,\ldots,w_{m(n)}\}$, and  a vector $\overline{p}=\overline{p}(n)=(p_1,\ldots,p_{m(n)})$ such that $p_i\in (0;1)$, for each $1\le i\le m$. To each vertex  $v\in\V$ we attribute a set of its features $W(v)\subseteq \W$ such that for each $i$, $1\le i\le m$, $w_i\in W(v)$ with probability $p_i$ independently of all other features and vertices. If $w\in W(v)$ then we say that $v$ {\em has chosen }$w$. Any two vertices $v,v'\in\V$ are connected by an edge in $\Gnmp$ if $W(v)$ and $W(v')$ intersect. If $\overline{p}(n)=(p,\ldots,p)$ for some $p\in (0;1)$ then $\Gnmp$ is a random intersection graph defined in \cite{GpSubgraph}. We denote it by $\Gnmpp$.

The random intersection graph model is very 
flexible and its properties change a lot if we alter the parameters. For example $\Gnmpp$ for some ranges of parameters behaves similarly to a random graph with independent edges (see \cite{GpEquivalence, GpEquivalence2}) but in some cases it exhibit large dependencies between edge appearance (see for example \cite{GpSubgraph,GpSubgraphPoisson}). It was proved in \cite{GpCoupling} that in both cases $\Gnmpp$  may be coupled with a random graph with independent edges so that with probability tending to $1$ as $n\to\infty$, $\Gnmpp$ is an overgraph of a graph with independent edges. It is also explained how this coupling may be used to obtain sharp results on threshold functions for $\Gnmpp$. Such  properties as connectivity, a Hamilton cycle containment or a matching containment are given as examples. In general, the coupling technique provides a very elegant method to get bounds on threshold functions for random intersection graphs for a large class of properties.  

The result presented in \cite{GpCoupling} is not sharp for some values of $n$, $m$ and $p$. Therefore it cannot be straightforward generalised to $\Gnmp$ with arbitrary $\overline{p}(n)$. In particular the method does not give sharp results for $np$ tending to a constant. In this article we  modify and extend the techniques used in \cite{GpCoupling} in order to overcome these constraints. First of all, to get the general result, we couple $\Gnmp$ with an auxiliary random graph which does not have fully independent edges. Therefore we need to prove some additional facts about the auxiliary random graph model. Moreover we need sharp bounds on the minimum degree threshold function for $\Gnmp$. Due to edge dependencies, estimation of moments of the random variable counting vertices with a given degree in $\Gnmp$ is complicated. Therefore we suggest a different approach to resolve the problem. We divide $\Gnmp$ into subgraphs so that the solution of a coupon collector problem combined with the method of moments provide the answer. The new approach to the coupling method allows to obtain better results on threshold functions for $\Gnmpp$ and by this means resolve open problems left over in \cite{GpCoupling}.   
 
Concluding, we provide a general method to establish bounds on threshold functions for many properties for $\Gnmp$. By means of the method we are able to obtain sharp thresholds for $k$--connectivity, perfect matching containment and hamiltonicity for the general model. Last but not least we considerably improve known results concerning $\Gnmpp$.

All limits in the paper are taken as $n\rightarrow \infty$.
Throughout the paper we use standard asymptotic notation
$o(\cdot)$, $O(\cdot)$, $\Omega(\cdot)$, $\Theta(\cdot)$, $\sim$, $\ll$, and $\gg$ defined as in \cite{KsiazkaJLR}. By $\Bin{n}{p}$ and $\Po{\lambda}$ we denote the binomial distribution with parameters $n$, $p$ and the Poisson distribution with expected value $\lambda$, respectively. We also use the phrase ``{\whp}'' to say with probability tending to one as $n$ tends to infinity.
All inequalities hold for $n$ large enough. If it does not influence the reasoning, for clarity we omit $\lfloor\cdot\rfloor$ and $\lceil\cdot\rceil$.

\section{Main Results}\label{SectionTwierdzenie}

In the article we compare random intersection graph $\Gnmp$ with a sum of a random graph with independent edges $G_2(n,\phat_2)$ and a random graph $G_3(n,\phat_3)$ constructed on the basis of a random
\linebreak 3--uniform hypergraph with independent hyperedges. Generally, for any
\linebreak $\phat=\phat(n)\in [0;1]$ and $i=2,\ldots,n$, let $H_i(n,\phat)$ be an $i$--uniform hypergraph with the vertex set $\V$ in which each $i$--element subset of $\V$ is added to the hyperedge set independently with probability $\phat$. $G_i(n,\phat)$ is a graph with the vertex set $\V$ and $\{v,v'\}$, $v,v'\in\V$, is an edge in $G_i(n,\phat)$ if there exists a hyperedge in $H_i(n,\phat)$ containing $v$ and $v'$.

We consider monotone graph properties of random graphs. 
For the family $\mathcal{G}$ of all graphs with the vertex set $\V$, we  call $\mathcal{A}\subseteq \mathcal{G}$ a property if it is closed under isomorphism. Moreover $\mathcal{A}$ is increasing if $G\in \mathcal{A}$ implies $G'\in \mathcal{A}$ for all $G'\in\mathcal{G}$ such that $E(G)\subseteq E(G')$. Examples of increasing properties are: $k$--connectivity, containing a perfect matching and containing a Hamilton cycle.

Let $\overline{p}=(p_1,\ldots,p_m)$ be such that $p_i\in (0,1)$, for all $1\le i\le m$. Define
\begin{equation}\label{RownanieS}
\begin{split}
S_1&=\sum_{i=1}^{m}np_i\left(1-(1-p_i)^{n-1}\right);\\
S_2&=\sum_{i=1}^{m}np_i\left(1-\frac{1-(1-2p_i)^{n}}{2np_i}\right);\\
S_3&=\sum_{i=1}^{m}np_i\left(\frac{1-(1-2p_i)^{n}}{2np_i}-(1-p_i)^{n-1}\right);\\
S_{1,t}&=\sum_{i=1}^{m}t\binom{n}{t}p_i^t(1-p_i)^{n-t}, \text{ for }t=2,3,\ldots,n.
\end{split}
\end{equation}
The following theorem is an extension of the result obtained in \cite{GpCoupling}. 

\begin{thm}\label{LematCoupling}
Let $\overline{p}=(p_1,\ldots,p_m)$ be such that $p_i\in (0,1)$, for all $1\le i\le m$ and $S_1$, $S_2$, and $S_3$  be given by \eqref{RownanieS}.
For a function $\omega$ tending to infinity let
\begin{equation}\label{RownanieHatp}
\begin{split}
\hat{p}&=
\frac{S_2-\omega\sqrt{S_2}-2S_2^2n^{-2}}{2\binom{n}{2}}
;\\
\hat{p}_2&=
\begin{cases} \frac{S_1-3S_3-\omega\sqrt{S_1}-2S_1^2n^{-2}}{2\binom{n}{2}},
&\text{for }S_3\gg \sqrt{S_1}
\text{ and }
\omega\ll S_3/\sqrt{S_1};\\
\frac{S_1-\omega\sqrt{S_1}-2S_1^2n^{-2}}{2\binom{n}{2}},
&\text{for }S_3=O(\sqrt{S_1});
\end{cases}
\\
\hat{p}_3&=
\begin{cases}
\frac{S_3-\omega\sqrt{S_1}-6S_3^2n^{-3}}{\binom{n}{3}},
\hspace{0.7cm}&
\text{for }S_3\gg \sqrt{S_1}
\text{ and }
\omega\ll S_3/\sqrt{S_1};\\
0, &\text{for }S_3=O(\sqrt{S_1}).
\end{cases}
\end{split}
\end{equation}
If $S_1\to \infty$ and $S_1=o\left(n^2\right)$ then for any increasing property $\mathcal{A}$. 
\begin{align}
\label{RownanieCoupling1}
\liminf_{n\to\infty}\Pra{\Gn{\hat{p}}\in \mathcal{A}}\le \limsup_{n\to\infty}\Pra{\Gnmp\in \mathcal{A}},\\
\label{RownanieCoupling2}
\liminf_{n\to\infty}\Pra{\Gn{\hat{p}_2}\cup\Gh{\hat{p}_3}\in\mathcal{A}}
\le
\limsup_{n\to\infty}\Pra{\Gnmp\in \mathcal{A}}.
\end{align}
\end{thm}

\begin{rem}
Assumption $S_1\to \infty$ is natural since for
 $S_1=o(1)$ {\whp} $\Gnmp$ is an empty graph. 
\end{rem}
\begin{rem}
$S_3$ is the expected number of edges in $\Gh{\hat{p}_3}$.
If\linebreak $S_3=O(\sqrt{S_1})$ then by Markov's inequality {\whp} the number of edges in $\Gh{\hat{p}_3}$ is at most $\omega\sqrt{S_1}$. Thus $S_2=S_1-S_3=S_1+O(\omega\sqrt{S_1})$ and the bound provided by \eqref{RownanieCoupling1} is as good as the one taking into consideration the edges from $\Gh{\hat{p}_3}$. 
\end{rem}
\begin{rem}
Theorem is also valid for $S_1=\Omega(n^2)$ but with  
\begin{equation*}
\begin{split}
\hat{p}&=1-\exp\left(-\frac{S_2-\omega\sqrt{S_2}}{2\binom{n}{2}}\right);\\
\hat{p}_2&=
\begin{cases} 1-\exp\left(-\frac{S_1-3S_3-\omega\sqrt{S_1}}{2\binom{n}{2}}\right),
&\text{for }S_3\gg \sqrt{S_1}
\text{ and }
\omega\ll S_3/\sqrt{S_1};\\
1-\exp\left(-\frac{S_1-\omega\sqrt{S_1}}{2\binom{n}{2}}\right),
&\text{for }S_3=O(\sqrt{S_1});
\end{cases}\\
\hat{p}_3&=
\begin{cases} 1-\exp\left(-\frac{S_3-\omega\sqrt{S_1}}{\binom{n}{3}}\right),
\hspace{0.7cm}&\text{for }S_3\gg \sqrt{S_1}
\text{ and }
\omega\ll S_3/\sqrt{S_1};\\
0,&\text{for }S_3=O(\sqrt{S_1}).
\end{cases}
\end{split}
\end{equation*}
\end{rem}

Denote by $\mathcal{C}_k$, $\mathcal{PM}$ and $\mathcal{HC}$ the following graph properties: a graph is $k$--connected, has a perfect matching and has a Hamilton cycle, respectively.  We will use Theorem~\ref{LematCoupling} to establish threshold functions for $\mathcal{C}_k$, $\mathcal{PM}$ and $\mathcal{HC}$ in $\Gnmp$. By $\mathcal{C}_k$ we denote the vertex connectivity. From the proof it follows that the threshold function for the edge connectivity is the same as this for $\mathcal{C}_k$.

For any sequence $c_n$ with a limit we write
\begin{equation}\label{RownanieFcn}
f(c_n)=
\begin{cases}
0&\text{for }c_n\to -\infty;\\
e^{-e^{-c}} &\text{for }c_n\to c\in (-\infty;\infty);\\
1&\text{for }c_n\to \infty.
\end{cases}
\end{equation}

\begin{thm}\label{TwierdzenieSpojnosc}
Let $\max_{1\le i\le m}p_i=o((\ln n)^{-1})$ and $S_1$ and $S_{1,2}$ be given by \eqref{RownanieS}.
\begin{itemize}
\item[(i)] If 
$
S_1=n(\ln n+ c_n), 
$
then
$$
\lim_{n\to\infty}\Pra{\Gnmp\in \mathcal{C}_1}=f(c_n),
$$
where $f(c_n)$ is given by \eqref{RownanieFcn}. 
\item[(ii)] Let $k$ be a positive integer and $a_n=\frac{S_{1,2}}{S_1}$. If 
$$
S_1=n(\ln n+(k-1)\ln\ln n + c_n), 
$$
then
$$
\lim_{n\to\infty}\Pra{\Gnmp\in\mathcal{C}_k}=
\begin{cases}
0&\text{ for }c_n\to -\infty\text{ and }a_n\to a\in (0;1];\\
1&\text{ for }c_n\to \infty.
\end{cases}
$$ 
\end{itemize}
\end{thm}
\smallskip
Assumption  $\max_{1\le i\le m}p_i=o((\ln n)^{-1})$ is necessary to avoid awkward cases. The problem is explained in more detail in Section~\ref{SectionStopnie}.
A straightforward corollary of the above theorem is that for $S_1=n(\ln n+c_n)$, $c_n\to  -\infty$ and any $k=1,2,\ldots,n$.
$$
\lim_{n\to\infty}\Pra{\Gnmp\in \mathcal{C}_k}=0\quad\text{and}\quad\lim_{n\to\infty}\Pra{\Gnmp\in \mathcal{HC}}=0.
$$

\begin{thm}\label{TwierdzeniePM}
Let $\max_{1\le i\le m}p_i=o((\ln n)^{-1})$ and $S_1$ be given by \eqref{RownanieS}. If 
$S_1=n(\ln n+c_n)$
then
$$
\lim_{n\to\infty}\Pra{\mathcal{G}(2n,m,\overline{p}(2n))\in\mathcal{PM}}=f(c_{2n}),
$$
where $f(\cdot)$ is given by \eqref{RownanieFcn}.
\end{thm}

\begin{thm}\label{TwierdzenieHamilton}
Let $\max_{1\le i\le m}p_i=o((\ln n)^{-1})$, $S_1$ and $S_{1,2}$ be given by \eqref{RownanieS} and $a_n=\frac{S_{1,2}}{S_1}$. If
$
S_1=n(\ln n+\ln\ln n + c_n), 
$
then
$$
\lim_{n\to\infty}\Pra{\Gnmp\in\mathcal{HC}}=
\begin{cases}
0&\text{ for }c_n\to -\infty\text{ and }a_n\to a\in (0;1];\\
1&\text{ for }c_n\to \infty.
\end{cases}
$$
\end{thm}
\noindent Already simple corollaries of Theorems~\ref{TwierdzenieSpojnosc}--\ref{TwierdzenieHamilton} give sharp threshold functions for $\Gnmpp$. For example.

\begin{cor}
Let $m\gg \ln^2 n$ and 
$p(1-(1-p)^{n-1})=(\ln n+c_n)/m$.
Then
$$
\lim_{n\to\infty}\Pra{\Gnmpp\in \mathcal{C}_1}=f(c_n)
$$
and
$$
\lim_{n\to\infty}\Pra{\mathcal{G}(2n,m,p)\in\mathcal{PM}}=f(c_{2n}),
$$
where $f(\cdot)$ is given by \eqref{RownanieFcn}.
\end{cor}
\noindent In particular we may state the following extension of the result from \cite{SingerPhD}.
\begin{cor}
Let $b_n$ be a sequence, $\beta$ and $\gamma$ be constants such that 
$
\beta\gamma(1-e^{-\gamma})=1.
$
If 
$$
m=
\beta n\ln n
\quad\text{
and}\quad
p=\frac{\gamma}{n}\left(1+\frac{b_n}{\ln n}\right)\quad\text{ then }
$$
$$
\lim_{n\to\infty}\Pra{\Gnmpp\in\mathcal{C}_1}
=f\left(\left(1+\frac{e^{-\gamma}\gamma}{1-e^{-\gamma}}\right)b_n\right)
$$
and
$$
\lim_{n\to\infty}\Pra{\mathcal{G}(2n,m,p(2n))\in\mathcal{PM}}
=f\left(\left(1+\frac{e^{-\gamma}\gamma}{1-e^{-\gamma}}\right)b_{2n}\right),
$$
where $f(\cdot)$ is given by \eqref{RownanieFcn}.
\end{cor}
Sometimes the method of the proof enables to improve results concerning $\Gnmpp$ even more. 
\begin{thm}\label{TwierdzenieHamiltonGp}
Let $m\gg \ln^2 n$
and 
\begin{equation}
\label{RownanieHamiltonGp}
p(1-(1-p)^{n-1})=
\frac{\ln n +\ln\left(\max\left\{1,\left(\frac{ npe^{-np}\ln n}{1-e^{-np}}\right)\right\}\right) + c_n}{m}.
\end{equation}
Then
$$
\lim_{n\to\infty}\Pra{\Gnmpp\in\mathcal{HC}}=
\begin{cases}
0&\text{ for }c_n\to -\infty;\\
1&\text{ for }c_n\to \infty.
\end{cases}
$$
\end{thm}

\begin{thm}\label{TwierdzenieSpojnosckGp}
Let $m\gg \ln ^2 n$, $k$ be a positive integer, and
$$
a_n=(np)^{k-1} \left( \left(\frac{e^{-np}\ln n}{1-e^{-np}}\right)^{k-1}+\frac{e^{-np}\ln n}{1-e^{-np}}\right)
$$
If 
\begin{equation*}
p(1-(1-p)^{n-1})
=
\frac{\ln n + \ln\left(\max\left\{1,a_n\right\}\right) + c_n}{m},
\end{equation*}
then
$$
\lim_{n\to \infty}\Pra{\Gnmpp\in \mathcal{C}_k}=
\begin{cases}
0&\text{ for }c_n\to -\infty;\\
1&\text{ for }c_n\to \infty.
\end{cases}
$$
\end{thm}
One of the question posed in \cite{GpCoupling} concerned the range of $m=m(n)$ for which the threshold function for $\mathcal{C}_k$ for $\Gnmpp$ coincides with this for $\delta(\Gnmpp)\ge 1$. Moreover we may ask when threshold function for $\mathcal{C}_k$ for $\Gnmpp$ is the same as this for $\mathcal{C}_k$ for $\Gn{\phat}$ with $\phat=mp^2$. Theorem~\ref{TwierdzenieSpojnosckGp} gives a final answer to these questions.
\begin{cor}
Let $k$ be a positive integer.
If 
\begin{equation*}
p(1-(1-p)^{n-1})=
\begin{cases}
\frac{\ln n + c_n}{m},
&\text{ for }\ln^2 n\ll m\ll \frac{n\ln n}{\ln\ln n}\\
\frac{\ln n + (k-1)\ln\ln n + c_n}{m},
&\text{ for }m=\Omega(n\ln n);
\end{cases}
\end{equation*}
  then
$$
\lim_{n\to \infty}\Pra{\Gnmpp\in\mathcal{C}_k}=
\begin{cases}
0&\text{ for }c_n\to -\infty;\\
1&\text{ for }c_n\to \infty.
\end{cases}
$$
\end{cor}

The remaining part of the article is organised as follows. Section~\ref{SectionCoupling} is devoted to a construction of the coupling of $\Gnmp$ and $\Gsuma$. The coupling is used in Section~\ref{SectionCoupling} to prove Theorem~\ref{LematCoupling}. Moreover in the subsequent sections the coupling will be applied in the proofs of the remaining theorems, especially Theorems~\ref{TwierdzenieHamiltonGp} and~\ref{TwierdzenieSpojnosckGp}. In Section~4 we give the threshold function for the minimum degree at least $k$ in $\Gnmp$. To this end we construct a coupling of a construction of $\Gnmp$ and a coupon collector process. Results included in Section~\ref{SectionStopnie} will be mainly used to prove $0$--statements of Theorems~\ref{TwierdzenieSpojnosc}--\ref{TwierdzenieSpojnosckGp}.  Section~5 is dedicated to the properties of $\Gsuma$. Results gathered in this section will be crucial in showing $1$-- statements of Theorems~\ref{TwierdzenieSpojnosc}--\ref{TwierdzenieSpojnosckGp}. In Section~6 we complete the proofs of Theorems~\ref{TwierdzenieSpojnosc}--\ref{TwierdzenieSpojnosckGp} using the results obtained in the previous sections.

\section{Coupling}\label{SectionCoupling}

In this section we present a proof of Theorem~\ref{LematCoupling}. In the proof we use auxiliary random graph models $\G{i}{M}$, $i=2,3,\ldots,n$, in which $M$ is
a random variable with non--negative integer values. For $i=2,\ldots,n$, $\G{i}{M}$ is constructed on the basis of a random hypergraph $\mathbb{H}_{*i}(n,M)$. $\mathbb{H}_{*i}(n,M)$ is a random hypergraph with the vertex set $\V$ in which a hyperedge set is constructed by sampling $M$ times with repetition elements from the set of all $i$--element subsets of $\V$ (all sets which are chosen several times are added only once to the hyperedge set). $\G{i}{M}$ is a graph with the vertex set $\V$ in which $v,v'\in\V$ are connected by an edge if $\{v,v'\}$ is contained in at least one of the hyperedges of $\mathbb{H}_{*i}(n,M)$. For simplicity of notation, if $M$ equals a constant $t$ with probability one or has the Poisson distribution, we write $\G{i}{t}$ or $\G{i}{\Po{\cdot}}$, respectively. Recall that similarly $G_i(n,\phat)$ is constructed on the basis of $H_i(n,\phat)$ -- a hypergraph with independent hyperedges (see definitions at the beginning of Section~\ref{SectionTwierdzenie}).

In this paper we treat random graphs as random variables. 
By a coupling $(\rg_1,\rg_2)$ of two random variables $\rg_1$ and $\rg_2$ we mean a choice of a probability space on which a random vector $(\rg_1',\rg_2')$ is defined and $\rg_1'$ and $\rg_2'$ have the same distributions as $\rg_1$ and $\rg_2$, respectively. For simplicity of notation we will not differentiate between $(\rg_1',\rg_2')$ and $(\rg_1,\rg_2)$.
For two random graphs $\rg_1$ and $\rg_2$ we write
$$
\rg_1\coupling \rg_2
$$
if
there exists a coupling $(\rg_1,\rg_2)$, such that in the probability space of the coupling $\rg_1$ is a subgraph of $\rg_2$ with probability $1$. 
Moreover, we write
$$
\rg_1=\rg_2,
$$
if $\rg_1$ and $\rg_2$ have the same probability distribution (equivalently there exists a coupling $(\rg_1,\rg_2)$ such that $\rg_1=\rg_2$ with probability one).
For two sequences of random graphs $\rg_1=\rg_1(n)$ and $\rg_2=\rg_2(n)$ we write
$$
\rg_1\coup \rg_2,
$$
if
there exists a sequence of couplings $(\rg_1(n),\rg_2(n))$, such that in the probability space of the coupling $\rg_1(n)$ is a subgraph of $\rg_2(n)$ with probability $1-o(1)$, respectively.

Note that, for any $\lambda$, in $\mathbb{H}_{*i}(n,\Po{\lambda})$ each edge appears independently with probability $1-\exp(-\lambda/{\textstyle \binom{n}{i}})$ (see \cite{GpEquivalence}). Thus
\begin{equation}\label{RownanieGstarPoisson}
\G{i}{\Po{\lambda}}=G_i\left(n,1-\exp\left(-\frac{\lambda}{\binom{n}{i}}\right)\right).
\end{equation}

We gather  here a few useful facts concerning couplings of random graphs. For proofs see~\cite{GpCoupling,GpEquivalence2}.
\begin{fact}\label{FaktCouplingGgwiazdka}
Let $M_n$ be a sequence of random variables and let $t_n$ be a sequence of positive integers.

(i) If
$
\Pra{M_n\ge t_n}=o(1)
$
then
$
\G{i}{M_n}\coup \G{i}{t_n}.
$

(ii) If
$
\Pra{M_n\le t_n}=o(1)
$
then
$
\G{i}{t_n}\coup \G{i}{M_n}.
$
\end{fact}
\begin{fact}\label{FaktCouplingIndependent}
Let $(\rg_i)_{i=1,\ldots,m}$ and $(\rg'_i)_{i=1,\ldots,m}$ be sequences of independent random graphs. If
$$
\rg_i\coupling \rg_i', \text{ for all }i=1,\ldots,m,
$$
then
$$
\bigcup_{i=1}^{m}\rg_i\coupling \bigcup_{i=1}^{m}\rg'_i.
$$
\end{fact}
\begin{fact}\label{FaktPrzechodniosc}
Let $\rg_1=\rg_1(n)$,  $\rg_2=\rg_2(n)$, and $\rg_3=\rg_3(n)$ be random graphs. If
$$
\rg_1\coup \rg_2\quad\text{and}\quad \rg_2\coup \rg_3
$$
then
$$
\rg_1\coup \rg_3.
$$
\end{fact}
\begin{fact}\label{FaktCouplingPrawdopodobienstwa}
Let $\rg_1=\rg_1(n)$ and $\rg_2=\rg_2(n)$ be two random graphs, such that
\begin{equation}\label{RownanieG1G2}
\rg_1\coup \rg_2.
\end{equation}
Then for any increasing property $\mathcal{A}$
$$
\liminf_{n\to\infty}\Pra{ \rg_1(n)\in \mathcal{A}}\le \limsup_{n\to\infty}\Pra{\rg_2(n)\in \mathcal{A}}.
$$
\end{fact}
\begin{proof}
Define event
$
\mathcal{E}:=\{\rg_1\subseteq \rg_2\}
$
on a probability space of a coupling $(\rg_1,\rg_2)$ existing by~\eqref{RownanieG1G2}. 
Then for any increasing property $\mathcal{A}$
\begin{align*}
\Pra{\rg_2\in \mathcal{A}}
&\ge\Pra{ \rg_2\in \mathcal{A}|\mathcal{E}\}\Pr\{\mathcal{E}}
\\
&\ge\Pra{ \rg_1\in \mathcal{A}|\mathcal{E}\}\Pr\{\mathcal{E}}
\\
&=\Pra{ \{\rg_1\in \mathcal{A}\}\cap\mathcal{E}}\\
&=\Pra{ \rg_1\in \mathcal{A}}+\Pra{ \mathcal{E}}-\Pra{ \{\rg_1\in \mathcal{A}\}\cup\mathcal{E}}\\
&\ge \Pra{ \rg_1\in \mathcal{A}}+\Pra{\mathcal{E}}-1\\
&= \Pra{ \rg_1\in \mathcal{A}}+o(1).\\
\end{align*}
The result follows by taking $n\to \infty$
\end{proof}

\begin{proof}[Proof of Theorem~\ref{LematCoupling}]
We will show only \eqref{RownanieCoupling2} in the case $S_3\gg \sqrt{S_1}$. The remaining cases follow by similar arguments. Recall that $S_2=S_1-S_3$ and $S_2=\Theta(S_1)$.

Let $w_i\in\W$. Denote by $V_i$ the set of vertices which have chosen feature $w_i$ (i.e.~$V_i=\{v\in\V: w_i\in W(v)\}$). Let
\begin{equation}\label{RownanieVprim}
V'_i=
\begin{cases}
V_i&\text{ for }|V_i|\ge 2;\\
\emptyset&\text{ otherwise}.
\end{cases}
\end{equation}
For each $i$, let  $\mathcal{G}[V'_i]$ be a graph with the vertex set~$\V$ and an edge set containing those edges from $\Gnmp$ which have both ends in $V'_i$ (i.e. its edges form a clique with the vertex set $V'_i$). In the proof we will use the fact that $\bigcup_{i=1}^{m}\mathcal{G}[V'_i]=\Gnmp$. First, for each $i$ we will construct a coupling of $\mathcal{G}[V'_i]$ with an auxiliary random graph and then we will couple a sum of the auxiliary random graphs with $\Gsuma$.

For each $i$, $1\le i\le m$, let
\begin{align}
X_i&=|V_i|;\nonumber\\
Y_i&=|V'_i|;\label{RownanieYi} \\
\nonumber Z_i&=\mathbb{I}_{\{Y_i \text{ is odd}\}},
\end{align}
where $\mathbb{I}_A$ is an indicator random variable of the event $A$. Note that $X_i$, $1\le i\le m$, are independent  random variables with the binomial distributions $\Bin{n}{p_i}$, $1\le i\le m$. 
For each $i$, $1\le i\le m$, we construct independently a coupling of $\G{2}{\frac{Y_i-3Z_i}{2}}\cup \G{3}{Z_i}$ and $\mathcal{G}[V'_i]$. Given $Y_i=y_i$ and $Z_i=z_i$, for each $i$ independently, we generate  instances of $\G{2}{\frac{y_i-3z_i}{2}}$ and  $\G{3}{z_i}$. Let $Y_i'=y_i'$ be the number of non-isolated vertices in the constructed instance of $\G{2}{\frac{y_i-3z_i}{2}}\cup \G{3}{z_i}$. By definition $y_i'\le y_i$. Set now $V'_i$ to be the union of the set of non--isolated vertices of the constructed instance of $\G{2}{\frac{y_i-3z_i}{2}}\cup \G{3}{z_i}$ and $y_i-y_i'$ vertices chosen uniformly at random from the remaining vertices. This coupling implies
$$
\G{2}{\frac{Y_i-3Z_i}{2}}\cup \G{3}{Z_i}\coupling\mathcal{G}[V_i].
$$

 Now 
let 
$$
M_2=\sum_{1\le i\le m} \frac{Y_i-3Z_i}{2}\quad\text{and}\quad M_3=\sum_{1\le i\le m} Z_i.
$$
Graphs $\G{2}{\frac{Y_i-3Z_i}{2}}\cup \G{3}{Z_i}$, $1\le i \le m$, are independent and $\mathcal{G}[V_i]$, $1\le i \le m$, are independent. Therefore by Fact~\ref{FaktCouplingIndependent}, the definition of $\Gnmp$, and the definitions of $\G{2}{\cdot}$ and $\G{3}{\cdot}$ we have
\begin{equation}\label{RownanieGgwiazdkaGw}
\begin{split}
\G{2}{M_2}
\cup 
\G{3}{M_3}
&=
\bigcup_{1\le i \le m}\left(\G{2}{\textstyle{\frac{Y_i-3Z_i}{2}}}\cup \G{3}{Z_i}\right)\\
&\coupling
\bigcup_{1\le i \le m}\mathcal{G}[V_i]\\
&=\Gnmp.
\end{split}
\end{equation}
Note that by definition
\begin{align*}
\E \sum_{i=1}^{m}Y_i &= \E \sum_{i=1}^{m}(X_i-\mathbb{I}_{X_i=1})=S_1
\quad
\text{and}\\
\E \sum_{i=1}^{m}Z_i &= \E \sum_{i=1}^{m}(\mathbb{I}_{X_i\,is\,odd}-\mathbb{I}_{X_i=1})=S_3,
\end{align*}
where $S_1$ and $S_3$ are defined by \eqref{RownanieS}. 
Therefore
$$
\E{ M}_2=\frac{S_1-3S_3}{2}\quad\text{and}\quad\E{ M}_3 = S_3.
$$
We will prove that $M_2$ and $M_3$ {\whp} are concentrated around their expected values.  
Since for each $i$, $1\le i\le m$, $X_i$ is a random variable with the binomial distribution $\Bin{n}{p_i}$ and $\E (X_i\mathbb{I}_{X_i=1})=\E \mathbb{I}_{X_i=1}$ we have
\begin{align*}
\Var \sum_{i=1}^{m}Y_i
&=\sum_{i=1}^{m}\Var \left(X_i-\mathbb{I}_{X_i=1}\right)
\\ 
&=
\sum_{i=1}^{m}
\left(
\Var X_i + \Var \mathbb{I}_{X_i=1} - 2(\E (X_i\mathbb{I}_{X_i=1})-\E X_i\E \mathbb{I}_{X_i=1})  
\right)
\\
&\le
\sum_{i=1}^{m}
\left(
\E X_i - \E \mathbb{I}_{X_i=1} + 2\E X_i\E \mathbb{I}_{X_i=1}
\right)\\
&=
\sum_{i=1}^{m}
\left(
\E Y_i + 2(np_i)^2(1-p_i)^{n-1}
\right)\\
&\le
\sum_{i=1}^{m}
\left(
\E Y_i + 3(np_i-np_i(1-p_i)^{n-1})
\right)\\
&= 4S_1.
\end{align*}
In the second last line we use the fact that function $(2nx+3)(1-x)^{n-1}$ is decreasing for $x\in [0,1)$. Similarly 
\begin{align*}
\Var \sum_{i=1}^{m}Z_i 
&=\sum_{i=1}^{m} \Var\left(\mathbb{I}_{X_iodd}-\mathbb{I}_{X_i=1}\right)
\\
&=\sum_{i=1}^{m}
\left(
\Var \mathbb{I}_{X_i\,odd} + \Var \mathbb{I}_{X_i=1} - 2(\E (\mathbb{I}_{X_i\,odd}\mathbb{I}_{X_i=1})-\E \mathbb{I}_{X_i\,odd}\E \mathbb{I}_{X_i=1})  
\right)
\\
&\le
\sum_{i=1}^{m}
\left(
\E \mathbb{I}_{X_i\,odd} - \E \mathbb{I}_{X_i=1} + 2\E \mathbb{I}_{X_i\,odd}\E \mathbb{I}_{X_i=1}
\right)\\
&\le
\sum_{i=1}^{m}
\left(
\E X_i - \E \mathbb{I}_{X_i=1} + 2\E X_i\E \mathbb{I}_{X_i=1}
\right)\\
&\le 4S_1.
\end{align*}
Therefore by Chebyshev's inequality, for any function $\omega'$ tending to infinity {\whp}
\begin{align}
\label{RownanieYikoncentracja}
\Pra{
\left|\sum_{i=1}^{m}Y_i - S_1\right|\ge \omega'\sqrt{S_1}
}\le \frac{4}{(\omega')^2}=o(1),
\\
\nonumber
\Pra{
\left|\sum_{i=1}^{m}Z_i - S_3\right|\ge \omega'\sqrt{S_1}
}\le \frac{4}{(\omega')^2}=o(1).
\end{align}
Thus with {\whp} for $\omega'\ll S_3/\sqrt{S_1}$ (the assumption is to ensure that for large $n$ r.h.s. is positive)
\begin{align*}
M_2&\ge \frac{S_1-3S_3-4\omega'\sqrt{S_1}}{2},\\ 
M_3&\ge S_3-\omega'\sqrt{S_1}.
\end{align*}
Therefore by Fact~\ref{FaktCouplingGgwiazdka} and \eqref{RownanieGgwiazdkaGw}
\begin{multline*}\label{Rownanienpduze1}
\G{2}{\frac{S_1-3S_3-4\omega'\sqrt{S_1}}{2}}\cup\G{3}{S_3-\omega'\sqrt{S_1}}\\
\coup \G{2}{M_2}\cup\G{3}{M_3}\coupling \Gnmp. 
\end{multline*}
We may assume that in the above coupling $\G{2}{\frac{S_1-3S_3-4\omega'\sqrt{S_1}}{2}}$ and $\G{3}{S_3-\omega'\sqrt{S_1}}$ are independent. The main reason for this is the fact that even though $M_2$ and $M_3$ are dependent (i.e. also $\G{2}{M_2}$ and $\G{3}{M_3}$ are dependent), the choices of hyperedges of $\mathbb{H}_{*i}(n,\cdot)$ in distinct draws are independent. 
Moreover note that in the coupling, in order to get $\G{2}{M_2}\cup\G{3}{M_3}$  from a sum of independent graphs $\G{2}{\frac{S_1-3S_3-4\omega'\sqrt{S_1}}{2}}\cup\G{3}{S_3-\omega'\sqrt{S_1}}$ we may proceed in the following way. Given independent instances of $\mathbb{H}_{*2}(n,{\frac{S_1-3S_3-4\omega'\sqrt{S_1}}{2}})$ and $\mathbb{H}_{*3}(n,{S_3-\omega'\sqrt{S_1}})$, $M_2=m_2$ and $M_3=m_3$:

\medskip
-- if $m_2$ ( or $m_3$ resp. )  is larger than $(S_1-3S_3-4\omega'\sqrt{S_1})/2$ ( or $S_3-\omega'\sqrt{S_1}$ resp. ) then we make $m_2-(S_1-3S_3-4\omega'\sqrt{S_1})/2$ (or $m_3-(S_3-\omega'\sqrt{S_1})$) additional draws and add hyperedges to  $\mathbb{H}_{*2}(n,\frac{S_1-3S_3-4\omega'\sqrt{S_1}}{2})$ 
( $\mathbb{H}_{*3}(n,S_3-\omega'\sqrt{S_1})$ resp. )

-- if $m_2$ ($m_3$ resp.) is smaller than $(S_1-3S_3-4\omega'\sqrt{S_1})/2$ ( or $S_3-\omega'\sqrt{S_1}$ resp. ) then we delete from $\mathbb{H}_{*2}(n,\frac{S_1-3S_3-4\omega'\sqrt{S_1}}{2})$ 
( $\mathbb{H}_{*3}(n,S_3-\omega'\sqrt{S_1})$ resp.) hypredges  attributed to the last draws to get exactly $m_i$, $i=2,3$,  draws.\\

Let $M_2'$ and $M_3'$ be independent random variables with the Poisson distributions  
$$
\Po{\frac{S_1-3S_3-5\omega'\sqrt{S_1}}{2}}\quad\text{and}\quad\Po{S_3-2\omega'\sqrt{S_1}},\text{ respectively.}
$$
Then by a sharp concentration of the Poisson distribution 
\begin{align*}
\Pra{M_2'\le \frac{S_1-3S_3-4\omega'\sqrt{S_1}}{2}}&=1-o(1)
\quad
\text{and}
\\
\Pra{M_3'\le S_3-\omega'\sqrt{S_1}}&=1-o(1).
\end{align*}
Therefore by Fact~\ref{FaktCouplingGgwiazdka}
and \eqref{RownanieGstarPoisson}
\begin{align*}
&\Gn{1-e^{\left(-\frac{S_1-3S_3-5\omega'\sqrt{S_1}}{2\binom{n}{2}}\right)}}
\cup
\Gh{1-e^{\left(-\frac{S_3-2\omega'\sqrt{S_1}}{\binom{n}{3}}\right)}}\\
&=
\G{2}{M'_2}\cup\G{3}{M'_3}\\
&\coup 
\G{2}{\frac{S_1-3S_3-4\omega'\sqrt{S_1}}{2}}\cup\G{3}{S_3-\omega'\sqrt{S_1}}\\
&\coup 
\Gnmp .
\end{align*}
For $S_1=o(n^2)$ (then also $S_3\le S_1=o(n^3)$), $\omega=5\omega'$, and $\phat_2$ and $\phat_3$ defined by \eqref{RownanieHatp} we have  
\begin{align*}
\hat{p}_2&=\frac{S_1-3S_3-5\omega'\sqrt{S_1}-\frac{2S_1^2}{n^2}}{2\binom{n}{2}}
\le 
1-e^{\left(-\frac{S_1-3S_3-5\omega'\sqrt{S_1}}{2\binom{n}{2}}\right)};\\
\hat{p}_3&\le\frac{S_3-2\omega'\sqrt{S_1}-\frac{6S_3^2}{n^3}}{\binom{n}{3}}
\le 
1-e^{\left(-\frac{S_3-2\omega'\sqrt{S_1}}{\binom{n}{3}}\right)}.
\end{align*}
Therefore using
standard couplings of $\Gn{\cdot}$ and $H_3(n,\cdot)$ finally we get
\begin{align*}
&\Gn{\hat{p}_2}\cup\Gh{\hat{p}_3}
\\
&\coupling
\Gn{1-e^{\left(-\frac{S_1-3S_3-5\omega'\sqrt{S_1}}{2\binom{n}{2}}\right)}}\cup
\Gh{1-e^{\left(-\frac{S_3-2\omega'\sqrt{S_1}}{\binom{n}{3}}\right)}}\\
&=\G{2}{\Po{\frac{S_1-3S_3-5\omega'\sqrt{S_1}}{2}}}\cup\G{3}{\Po{S_3-2\omega'\sqrt{S_1}}}
\\
&\coup 
\Gnmp.
\end{align*}
The result follows by Facts~\ref{FaktPrzechodniosc} and~\ref{FaktCouplingPrawdopodobienstwa}.
\end{proof}

\section{Vertex degrees in $\Gnmp$}\label{SectionStopnie}
For any graph $G$ denote by $\delta(G)$ the minimum vertex degree in $G$. 
\begin{lem}\label{LematStopnie}
Let $c_n$ be a sequence of real numbers, $\bar{p}=(p_1,\ldots,p_m)$ be such that $\max_{1\le i\le m}p_i=o((\ln n)^{-1})$, and $S_1$ be given by \eqref{RownanieS}.\\
  If
$$
S_1=n(\ln n+ c_n)
$$
then
$$\lim_{n\to \infty}\Pra{\delta(\Gnmp)\ge 1}=f(c_n),$$
where $f(\cdot)$ is given by 
\eqref{RownanieFcn}.
\end{lem}
\noindent Note that the condition $\max_{1\le i\le m}p_i=o((\ln n)^{-1})$ is necessary. Otherwise the number of  vertices of a given degree depends more on the fluctuations of the values of the vector $\overline{p}$. For example let $m=n^{2}$ and $\overline{p}$ equals
\begin{align*}
\Big(\underbrace{\frac{b_n}{\ln n},\ldots,\frac{b_n}{\ln n}}_{\ln n(\ln n+c_n)/b_n}, \frac{1}{nm},\ldots,\frac{1}{nm}\Big)
\quad\text{ or }\\
\left(\frac{1}{\sqrt{\ln n}},\sqrt{\frac{\ln n+c_n}{nm}},\ldots,\sqrt{\frac{\ln n+c_n}{nm}}\right),
\end{align*} 
for some $b_n,c_n=o(\ln n)$.
In both cases 
$$
\sum_{i=1}^{m}np_i\left(1-(1-p_i)^{n-1}\right)=n(\ln n+c_n+o(1))
$$ 
but the expected number of vertices of degree $0$ in $\Gnmp$ is
$$
(1+o(1))\exp\left(-c_n-\frac{b_n}{2}\right)
\quad\text{ and }\quad
(1+o(1))\exp(-c_n), 
$$
respectively.
\begin{lem}\label{LematStopnie2}
Let $c_n$ be a sequence of real numbers, $k$ be a positive integer, $\bar{p}=(p_1,\ldots,p_m)$ be such that $\max_{1\le i\le m}p_i=o(\ln n^{-1})$,  
and $S_1$ and $S_{1,t}$, $t=2\ldots,k$, be given by \eqref{RownanieS}.
\begin{itemize}
\item[(i)] If 
$$
S_1=n\left(\ln n+(k-1)\ln \left(\max\left\{1, \left(\frac{S_{1,2}}{S_1}\ln n\right)\right\}\right) + c_n\right)
$$  then
$$
\lim_{n\to \infty}\Pra{\delta(\Gnmp)\ge k}=0
\quad
\text{for }c_n\to - \infty.
$$

\item[(ii)] If
$$
S_1-\sum_{t=3}^{k}S_{1,t}=n\left(\ln n+(k-1)\ln \left(\max\left\{1, \left(\frac{S_{1,2}}{S_1}\ln n\right)\right\}\right) + c_n\right)
$$
 then
$$
\lim_{n\to \infty}\Pra{\delta(\Gnmp)\ge k}=1
\quad
\text{for }c_n\to \infty.
$$
\end{itemize}  
\end{lem}
\noindent Here and in the proof we assume that $\sum_{t=3}^{2}S_{1,t}=0$. 

If we assume that $p_i=p$, for all $1\le i\le m$, then with a little more work the result of Lemma~\ref{LematStopnie2} may be improved.
\begin{lem}\label{LematStopniek}
Let $p\in (0;1)$, $k$ be a positive integer, $m=m(n)$ be such that  $m\gg \ln^2 n$,
and
$$
a_n=(np)^{k-1} \left( \left(\frac{e^{-np}\ln n}{1-e^{-np}}\right)^{k-1}+\frac{e^{-np}\ln n}{1-e^{-np}}\right).
$$
If  
\begin{equation*}
p(1-(1-p)^{n-1})
=
\frac{\ln n + \ln\left(\max\left\{1,a_n\right\}\right) + c_n}{m},
\end{equation*}
  then
$$
\lim_{n\to \infty}\Pra{\delta(\Gnmpp)\ge k}=
\begin{cases}
0&\text{ for }c_n\to -\infty;\\
1&\text{ for }c_n\to \infty.
\end{cases}
$$
\end{lem}

\begin{proof}[Proof of Lemma \ref{LematStopnie}]

In the proof we assume that $c_n=o(\ln n)$. Note that if $\overline{p}'=(p'_1,\ldots,p'_m)$ and $\overline{p}=(p_1,\ldots,p_m)$ are such that $p'_i\le p_i$, for all $1\le i\le m$, then $\Gnm{}{\overline{p}'}\coupling \Gnmp$ and $\sum_{i=1}^m np'_i(1-(1-p'_i)^{n-1})\le \sum_{i=1}^m np_i(1-(1-p_i)^{n-1})$. Therefore by monotonicity of the considered property, analysis of the case $c_n=o(\ln n)$ is enough to prove the general result stated above. 

Consider a coupon collector process in which in each draw one choose one coupon uniformly at random from $\V$. In order to determine the minimum degree
we establish a coupling of the coupon collector process on $\V$ and the construction of $\Gnmp$.
Define $V_i'$ and $Y_i$ as in \eqref{RownanieVprim} and \eqref{RownanieYi}. 
We consider a process in which we collect coupons from $\V$ and at the same time we construct a family of sets $\{V'_i:i=1,\ldots,m\}$ (i.e. equivalently we construct an instance of $\Gnmp$). Assume that we have a given vector
$$
(Y_1,Y_2,\ldots,Y_m)=(y_1,y_2,\ldots,y_m)
$$
chosen so that $Y_i$ are independent random variables with distribution of the random variables defined in \eqref{RownanieYi}. We divide the process of collecting coupons into $m$ phases. In the $i$--th phase, $1\le i\le m$, we draw independently one by one vertices uniformly at random from $\V$ until within the phase we get $y_i$ distinct vertices. Let $V'_i$ be the set of vertices chosen in the $i$--th phase. We construct an instance of $\Gnmp$ by connecting by edges all pairs of vertices within $V_i'$, for all $1\le i\le m$. After the $m$--th phase   $\bigcup_{1\le i\le m}V'_i$ is the set of non-isolated vertices in $\Gnmp$. 
Denote by $T_i$ the number of draws in the $i$--th phase. If $Y_i=y_i=0$ then $T_i=0$. If $Y_i=y_i\ge 2$ and in the $i$--th phase we have already collected $j<y_i$ vertices then the number of draws to collect the $j+1$--st vertex has the geometric distribution with parameter $\frac{n-j}{n}$. Therefore 
\begin{align*}
\E (T_i-Y_i|Y_i=y_i)
&=\left(\sum_{j=0}^{y_i-1}\frac{n}{n-j}\right)-y_i\\
&=\sum_{j=0}^{y_i-1}\frac{j}{n-j}\\
&\le 
\begin{cases}
\sum_{j=0}^{y_i-1}\frac{2j}{n}=\frac{y_i(y_i-1)}{n}&\text{for }2\le y_i<n/2;\\
n\ln n&\text{for }y_i\ge n/2.
\end{cases}
\end{align*}
Note that
$$\Pra{Y_i\ge \frac{n}{2}}\le
\binom{n}{\frac{n}{2}}p_i^{\frac{n}{2}}\le \left(\frac{enp_i}{\frac{n}{2}}\right)^{\frac{n}{2}}\le \frac{p_i^2}{\ln n}.$$
Thus
\begin{align*}
\E (T_i-Y_i)
&\le \sum_{y_i=2}^{n/2}\frac{y_i(y_i-1)}{n}\binom{n}{y_i}p_i^{y_i}(1-p_i)^{n-y_i}+n\ln n\Pra{Y_i\ge \frac{n}{2}}\\
&\le 2np_i^2.
\end{align*}
Therefore
\begin{align*}
\E \left(\sum_{i=1}^{m}(T_i-Y_i)\right)
&\le \sum_{i=1}^{m}2np_i^2\\
&=\sum_{i=1}^{m}2\max\left\{p_i,\frac{1}{n}\right\}\min\{np_i,n^2p_i^2\}\\
&\le 4\max\left\{p_1,\ldots,p_m,\frac{1}{n}\right\}
\sum_{i=1}^{m}\frac{1}{2}\min\{np_i,n^2p_i^2\}\\
&=o\left(\frac{1}{\ln n}\right) \sum_{i=1}^{m}np_i\left(1-(1-p_i)^{n-1}\right)\\
&=o\left(\frac{S_1}{\ln n}\right),
\end{align*}
where $S_1$ is defined by \eqref{RownanieS}. In order to get the last but one line consider two cases $p(n-1)\le 1/2$ and $p(n-1)> 1/2$.
 
Moreover by Markov's inequality for any $\omega$
$$
\Pra{
\sum_{i=1}^{m}(T_i-Y_i)\ge \frac{S_1}{\omega \ln n} 
}\le \E \left(\sum_{i=1}^{m}(T_i-Y_i)\right)\frac{\omega \ln n}{S_1}.
$$
Therefore {\whp} 
\begin{equation}\label{RownanieTY}
\sum_{i=1}^{m}(T_i-Y_i)\le \frac{S_1}{\omega \ln n}\quad
\text{, for any }1\ll \omega\ll \frac{S_1}{(\ln n\ \E (\sum_{i=1}^{m}(T_i-Y_i)))}.
\end{equation}
Given $1\ll\omega\ll\min \{S_1/(\ln n\ \E (\sum_{i=1}^{m}(T_i-Y_i)));\sqrt{S_1}/\ln n\}$ let
$$
T_-=S_1-\omega\sqrt{S_1}\quad\text{and}\quad T_+=S_1+\omega\sqrt{S_1}+\frac{S_1}{\omega \ln n}.$$
In the probability space of the coupling described above define events:
\medskip

\begin{tabular}{lcp{15cm}}
$\mathcal{A_-}$&--&all coupons are collected in at most $T_-$ draws;\\
$\mathcal{A_+}$&--&all coupons are collected in at most $T_+$ draws;\\
$\mathcal{A}$&--&$\delta(\Gnmp)\ge 1$; \\
$\mathcal{B}$&--&the construction of $\Gnmp$ is finished between $T_-$-th \\ 
&&and $T_+$-th draw;\\
$\mathcal{B}_1$&--&the construction of $\Gnmp$ is finished in at most\\
&& $\sum_{i=1}^m Y_i+S_1/(\omega \ln n)$ draws;\\
$\mathcal{B}_2$&--&in $\Gnmp$ we have $S_1-\omega\sqrt{S_1}\le \sum_{i=1}^m Y_i\le S_1+\omega\sqrt{S_1}$.\\
\end{tabular}
\medskip

\noindent
By definition
$$
\mathcal{A_+}\cap\mathcal{B}
\subseteq
\mathcal{A}\cap\mathcal{B}
\subseteq
\mathcal{A_-}\cap\mathcal{B}.
$$
Therefore
\begin{equation}\label{RownanieAplusminus}
\Pra{\mathcal{A_+}}-\Pra{\mathcal{B}^c}
\le
\Pra{\mathcal{A}}
\le \Pra{\mathcal{A_-}}+\Pra{\mathcal{B}^c},
\end{equation}
where for any event $\mathcal{C}$ we denote by $\mathcal{C}^c$ its complement. 

Let $S_1=n(\ln n + c_n)$, then
$$T_\pm=n(\ln n+c_n+o(1)).$$
Therefore by the classical results on the coupon collector problem
$$\Pra{\mathcal{A}_-}\sim\Pra{\mathcal{A}_+}\sim f(c_n)$$
and by \eqref{RownanieYikoncentracja} and \eqref{RownanieTY}
$$
\Pra{\mathcal{B}^c}\le\Pra{\mathcal{B}_1^c\cup \mathcal{B}_2^c}=o(1).
$$
Thus the lemma follows by substituting  the above values to \eqref{RownanieAplusminus}. 
\end{proof}

The proofs of Lemmas~\ref{LematStopnie2} and~\ref{LematStopniek} rely on a coupling with the coupon collector process similar to this used in the proof of Lemma~\ref{LematStopnie} and a technique of dividing $\Gnmp$ into auxiliary subgraphs. We will give here an idea of the proofs of Lemmas~\ref{LematStopnie2} and~\ref{LematStopniek} and leave purely technical details to Appendix.  

First we gather some simple facts concerning the division of $\Gnmp$. Let $k\ge 2$ be an integer.
For each $2\le t\le k$, let  $\Gnm{t}{\overline{p}}$ be a random graph with a vertex set $\V$ and an edge set consisting of those edges from $\Gnmp$, which are contained in at least one of the sets $\{V_i: |V_i|=t\}$, where $V_i$ is defined as in the proof of Theorem~\ref{LematCoupling}. Moreover let $\Gnm{k+1}{\overline{p}}$ be a subgraph of $\Gnmp$ containing only those edges which are subsets of at least one of the sets $\{V_i: |V_i|\ge k+1\}$. Let
$$
Q_t =\sum_{i=1}^{m}\mathbb{I}_{Y_i=t}, \text{ for all }t=2,\ldots,k,
\quad\text{and}\quad   Q_{k+1} =\sum_{i=1}^{m}Y_i-\sum_{t=2}^{k}tQ_t.
$$ 
Note that, for all $t=2,\ldots,k$, we have $\Gnm{t}{\overline{p}}=\G{t}{Q_t}$, where $\G{t}{\cdot}$ is defined as in Section~\ref{SectionCoupling}. Moreover
\begin{equation*}
\E Q_t =\frac{S_{1,t}}{t}
 \text{ for all }t=2,\ldots,k,\quad \text{and }\quad
 \E Q_{k+1} =S_1-\sum_{t=2}^{k}S_{1,t},
 \end{equation*}
 where  $S_1$ and $S_{1,t}$ are defined by \eqref{RownanieS}. For all $1\le t\le k$, 
$Q_t$ is a sum of independent Bernoulli random variables and $S_{1,t}\le S_1$. Therefore by Chebyshev's inequality, Markov's inequality, and by  \eqref{RownanieYikoncentracja} for any $\omega\to\infty$ {\whp}
\begin{align}
\label{RownanieMtkoncentracja}
\max\left\{0,\frac{S_{1,t}-\omega\sqrt{S_1}}{t}\right\}&\le Q_t\le \frac{S_{1,t}+\omega\sqrt{S_1}}{t}, \text{ for all }t=2,3,\ldots,k 
\\
\label{RownanieM3koncentracja}
Q_{k+1}&\ge S_1-\sum_{t=2}^{k}S_{1,t} - k\omega\sqrt{S_1},\\
\label{RownanieMtbrakkoncentracja}
Q_{k+1}+\sum_{j=2}^{k}jQ_{j}-tQ_t
&\le
S_1-S_{1,t}+2\omega\sqrt{S_1},
\text{ for all }t=2,3,\ldots,k.
\intertext{ Moreover, for any $t=2,3,\ldots,k$, 
if $S_{1,t}\to \infty$ then {\whp}}
\frac{S_{1,t}- \omega\sqrt{S_{1,t}}}{t} 
&\le
Q_t
\le
\frac{S_{1,t}+ \omega\sqrt{S_{1,t}}}{t}. 
\label{RownanieM2koncentracja}
\end{align}
Let $\omega$ be a function tending slowly to infinity. In the proofs we will assume that $\omega$ is small enough to get the needed bounds.
 
For $t=2,\ldots,k$ let 
\begin{equation}\label{RownanieTp}
\begin{split}
T_{t+}&
=
S_1-S_{1,t}+2\omega\sqrt{S_1}+\frac{S_1}{\omega \ln n};
\\
T_-&
=
\max\left\{0,S_1-\sum_{t=2}^{k}S_{1,t} - k\omega\sqrt{S_1},\right\};
\\
\phat_{t-}&
=
\frac{S_{1,t}-2\omega\sqrt{S_{1,t}}-t!S^2_{1,t}n^{-t}}{t\binom{n}{t}};
\\
\phat_{t+}&
=
\frac{S_{1,t}+2\omega\sqrt{S_{1,t}}}{t\binom{n}{t}};
\\
\hat{q}_{t}&
=
\begin{cases}
0,&\text{for }
S_{1,t}=O(\omega\sqrt{S_{1}});\\
\frac{S_{1,t}-2\omega\sqrt{S_{1}}-t!S_{1,t}^2n^{-2}}{t\binom{n}{t}},&\text{otherwise}.
\end{cases}
\end{split}
\end{equation}
For all $t=2,\ldots,k$, 
$\Gnm{t}{\overline{p}}=\G{t}{M_2}$, thus by \eqref{RownanieMtkoncentracja} and \eqref{RownanieM2koncentracja}, using the same methods as in the proof of Theorem~\ref{LematCoupling} we can show the following fact.
\begin{fact}\label{FaktStopnieCoupling}
Let $\hat{q}_t$ and $\phat_{t\pm}$ be defined as in \eqref{RownanieTp}. Then
\begin{align*}
G_t(n,\hat{q}_{t})&\coup\Gnm{t}{\overline{p}},
\end{align*}
where $G_t(n,\hat{q_t})$ is independent from ($\Gnm{j}{\overline{p}})_{j=2,\ldots k+1,j\neq t}$.\\
Moreover if $S_{1,t}\to\infty$ then
\begin{align*}
G_t(n,\phat_{t-})
&\coup\Gnm{t}{\overline{p}}
\coup G_t(n,\phat_{t+}).  
\end{align*}
where $G_t(n,\phat_{t\pm})$ is independent from ($\Gnm{j}{\overline{p}})_{j=2,\ldots k+1,j\neq t}$.
\end{fact}

As in the proof of Lemma~\ref{LematStopnie} we use the coupling of the coupon collector process on~$\V$ and the construction of $\Gnmp$. However here in the process we omit rounds in which $Y_i=t$, $2\le t\le k$. Therefore we construct $\bigcup_{j=2,\ldots,k+1;j\neq t}\Gnm{t}{\overline{p}}$ instead of $\Gnmp$. Reasoning in the same way as in the proof of Lemma~\ref{LematStopnie},  by the definition of $Q_t$, \eqref{RownanieTY}, and \eqref{RownanieMtbrakkoncentracja} we get that {\whp} the construction of $\bigcup_{j=2,\ldots,k+1;j\neq t}\Gnm{t}{\overline{p}}$ is finished before the $T_{t+}$--th draw. Similarly, if in the process we omit rounds
with $Y_i\le k$, then {\whp} the construction of $\Gnm{k+1}{\overline{p}}$ is not finished before the $T_-$ draw. Therefore an analogous reasoning to this used in the proof of Lemma~\ref{LematStopnie} gives the following facts.

\begin{fact}\label{FaktCollectorPlus}
There exists a coupling of the coupon collector process with the set of coupons $\V$ and the construction of  $\bigcup_{j=2,\ldots,k+1;j\neq t}\Gnm{j}{\overline{p}}$ such that {\whp} the number of collected coupons in $T_{t+}$ draws is at least the number of non--isolated vertices in $\bigcup_{j=2,\ldots,k+1;j\neq t}\Gnm{j}{\overline{p}}$.  
\end{fact}

\begin{fact}\label{FaktCollectorMinus}
There exists a coupling of the coupon collector process with the set of coupons $\V$ and the construction of $\Gnm{k+1}{\overline{p}}$ such that {\whp} the number of collected coupons in $T_-$ draws is at most the number of non--isolated vertices in $\Gnm{k+1}{\overline{p}}$.  
\end{fact}

If there exists a vertex which is isolated in $\bigcup_{t=3}^{k+1}\Gnm{t}{\overline{p}}$ and has degree at most $k-1$ in $\Gnm{2}{\overline{p}}$, then $\delta(\Gnmp)\le k-1$ and if each vertex in $\Gnm{k+1}{\overline{p}}$ is non-isolated or has degree at least $k$ in $\Gnm{2}{\overline{p}}$, then $\delta(\Gnmp)\ge k$. The lemmas follow by the above facts and by the first and the second moment method, which is used to count the vertices, which are isolated in a considered auxiliary graph or have degree at least or less than k in another auxiliary graphs.
Technical details of the proofs are postponed to Appendix.

\section{Structural properties of $\Gn{\hat{p}_2}\cup\Gh{\hat{p}_3}$}

In this section we give several structural results concerning $\Gn{\hat{p}_2}\cup\Gh{\hat{p}_3}$, which are analogous to those known for $\Gn{\hat{p}_2}$. Then we will show how combine those properties with the coupling constructed in the proof of Theorem~\ref{LematCoupling} in order to get the threshold functions of $\Gnmp$.  

For any graph $G$ and any set $S\subseteq V(G)$ denote by $N_G(S)$ the set of neighbours of vertices form $S$ contained in $V(G)\setminus S$. 
For simplicity we write  $N_2(\cdot)=N_{\Gn{\hat{p}_2}}(\cdot)$ and $N_3(\cdot)=N_{\Gsuma}(\cdot)$. Note that given $\Gsuma$ we have $|N_2(S)|\le |N_3(S)|$. 
%
%
%
%

\begin{lem}\label{LematWlasnosciGsuma}
Let $k$ and $C$ be positive integers and
\begin{align*}
\pdwa+\frac{n}{2}\ptrzy&=\frac{\theta_n\ln n}{n},
\quad
\text{where }
\liminf_{n\to\infty}\theta_n=\theta>\frac{1}{2}\text{ and }\theta_n=O(1),\\
\pdwa&=\frac{\theta'_n\ln n}{n},\quad
\text{where }
\liminf_{n\to\infty}\theta'_n=\theta'>\frac{1}{2}\text{ and }\theta'_n=O(1).\end{align*}
Let moreover $\gamma$ be a real number such that $1-\theta'<\gamma<1$.
Then {\whp} $\Gsuma$ has the following properties: 
\begin{itemize}
\item[(i)] $\mathcal{B}_1$ -- for all $S\subseteq \V$ such that $n^{\gamma}\le |S|\le \frac{1}{4}n$ 
$$
|N_3(S)|>2|S|.
$$
\item[(ii)] $\mathcal{B}_2$ -- for all $S\subseteq \V$ such that $n^{\gamma}\le |S|\le \frac{2}{3}n$ 
\begin{align*}
|N_3(S)|
&>\min\{|S|,4n\ln \ln n/\ln n\}\\
&>\min\{|S|,\alpha(\Gsuma)\},
\end{align*}
where $\alpha(\Gsuma)$ is the stability number of $\Gsuma$.
\item[(iii)] $\mathcal{B}_{3,k}$ -- for all $S\subseteq \V$ if $1\le |S|\le n^{\gamma}$ and all vertices in $S$ have degree at least $4k+14$ in $\Gsuma$ we have
$$
|N_3(S)|\ge 2k|S|.
$$
\item[(iv)] $\mathcal{B}_{4,C}$ -- any two vertices of degree at most $C$ are at distance at least $6$.
\item[(v)] $\mathcal{B}_5$ -- contains a path of length at least 
$$
\left(1-\frac{8\ln 2}{\ln n}\right)n.
$$
\end{itemize} 
\end{lem}

\begin{proof}
\noindent(i) and (ii)\  
By Chernoff's inequality (for the proof see for example Theorem~2.1 in \cite{KsiazkaJLR}) for any $\delta=\delta(n)=o(1)$ and any binomial random variable $X=X_n$ we have 
$$
\Pra{X\le \delta \E X}\le \exp(-\E X(\delta\ln \delta + 1 - \delta)) =\exp(-(1+o(1))\E X).
$$
First consider the case $n^{\gamma}\le |S|\le 4n\ln\ln n/\ln n$. 
Let $s=|S|$. Then
$|N_2(S)|$ has the binomial distribution 
$\Bin{n-s}{1-(1-\pdwa)^s}$ with the expected value $\E |N_2(S)|\gg 2|S|$.
Denote by $X$ the number of sets $S$ of cardinality\linebreak $n^{\gamma}\le |S|\le 4n\ln\ln n/\ln n$ such that $|N_2(S)|\le 2|S|$. Using Chernoff's inequality we get
\begin{align*}
\E X
\le&
\sum_{s=n^{\gamma}}^{4 n\ln\ln n/\ln n}
\binom{n}{s}\exp\left(-(1+o(1))(n-s)(1-(1-\pdwa)^s)\right)
\\
\le&
\sum_{s=n^{\gamma}}^{n/(\ln n\ln\ln n)}
\exp\left(s\left(1+\ln\frac{n}{s}\right)-(1+o(1))sn\pdwa \right)\\
&+
\sum_{s=n/(\ln n\ln\ln n)}^{4 n\ln\ln n/\ln n}
\exp\left(s\left(1+\ln\frac{n}{s}\right)-n\frac{\theta'}{2\ln \ln n}\right)\\
\le&
\sum_{s=n^{\gamma}}^{n/(\ln n\ln\ln n)}
\exp\left(s\left(1+(1-\gamma)\ln n-(1+o(1))\theta'\ln n\right)\right)\\
&+
\sum_{s=n/(\ln n\ln\ln n)}^{4 n\ln\ln n/\ln n}
\exp\left(s\left((1+o(1))\ln \ln n - \frac{\theta'\ln n}{8(\ln \ln n)^2}\right)\right)\\
=&\ o(1).
\end{align*}
Therefore {\whp} for all sets $S\subseteq \V$, $n^{\gamma}\le |S|\le 4n\ln\ln n/\ln n$, we have
$$
|N_2(S)|\ge 2|S|.
$$ 

Now consider the case $|S|\ge 4n\ln\ln n/\ln n$. Let $r=4n\ln\ln n/\ln n$. 
Let moreover $\overline{K}_r$ and  $\overline{K}_{r,r}$ be the complement of the complete graph on $r$ vertices and the complement of the complete bipartite graph with each set of bipartition of cardinality $r$. 
Denote by $X_r$ and $X_{r,r}$ the number of $\overline{K}_r$ and $\overline{K}_{r,r}$ in $\Gsuma$, respectively. Then
\begin{align*}
\E X_r
&=
\binom{n}{r}(1-\pdwa)^{\binom{r}{2}}(1-\ptrzy)^{\binom{r}{3}+\binom{r}{2}(n-r)}\\
&\le 
\left(
\frac{en}{r}\exp\left(-\frac{r}{2}\left(\pdwa+(n-r)\ptrzy\right)+o(1)\right)
\right)^r\\
&\le
\left(
\frac{e\ln n}{4\ln \ln n}\exp\left(-2\theta_n\ln \ln n\right)
\right)^r=o(1)
\end{align*}
and
\begin{align*}
\E X_{r,r}
&\le
\binom{n}{r}^2(1-\pdwa)^{r^2}(1-\ptrzy)^{2\binom{r}{2}r+r^2(n-2r)}\\
&\le 
\left(
\frac{en}{r}\exp\left(-\frac{r}{2}(\pdwa+(n-2r)\ptrzy)+o(1)\right)
\right)^{2r}=o(1).
\end{align*}
Therefore {\whp} $X_r=0$ which implies that {\whp} 
$$\alpha(\Gsuma)\le r.$$ 
Moreover {\whp} $X_{r,r}=0$ thus {\whp}
 for any $S\subseteq \V$ such that $r\le |S|$ we have 
$$
N_3(S)\ge n-|S|-r.
$$
(Otherwise $\Gsuma$ would contain $\overline{K}_{r,r}$.)\\
Therefore {\whp} for any $S\subseteq \V$ such that $r\le |S|\le 2n/3$
$$N_3(S) \ge \min\{|S|,r\}
\ge \min\{|S|,\alpha(\Gsuma)\}.
$$ 
Moreover {\whp} for any $S\subseteq \V$ such that $r\le |S|\le n/4$ we have
$$
N_3(S)\ge n-|S|-r\ge 2|S|.
$$
This finishes the proof of (i) and (ii).

\noindent
(iii)
For any two disjoint sets $S\subseteq \V$ and $S'\subseteq \V$ we denote by
 $e(S;S')$ the number of edges in $\Gsuma$ with one end in  $S$ and one end in $S'$ and by
 $e(S)$ the number of edges in $\Gsuma$ with both ends in $S$. We will bound the numbers $e(S)$ and $e(S,S')$ for $S\subseteq \V$ and $S'\subseteq \V\setminus S$ such that $1\le |S|\le n^{\gamma}$ and $|S'|=O(|S|)$. 
 Let $k$ be a positive integer, $S\subseteq \V$, $S'\subseteq \V\setminus S$, $|S|=s$, $1\le s\le n^{\gamma}$ and $|S'|=2ks$. 
A pair $\{v,v'\}$ ($v\in S$ and $v'\in S'$) is an edge in $\Gsuma$ if at least one of the three  following events occurs:  
\begin{itemize}
\item $\{v,v'\}$ is an edge in $\Gn{\pdwa}$, 
\item there is a hyperedge $\{v,v',v''\}$ with $v''\in \V\setminus(S\cup S')$ in $H_3(n,\ptrzy)$ associated with $\Gh{\ptrzy}$, 
\item there is a hyperedge $\{v,v',v''\}$ with $v''\in S\cup S'$ in $H_3(n,\ptrzy)$ associated with $\Gh{\ptrzy}$.
\end{itemize}
 The union of the two first events occurs  with probability at most $\pdwa+(n-(2k+1)s)\ptrzy=O(\ln n/n)$ independently for all $v\in S$ and $v'\in S'$. Moreover each hyperedge of $H_3(n,\ptrzy)$ with all vertices in $S\cup S'$ appears independently with probability $\ptrzy=O(\ln n/n^2)$ and generates two edges between $S$ and $S'$ in $\Gh{\ptrzy}$. Therefore 
\begin{align*}
&\Pra{
\exists_{\begin{subarray}{l} S,S'\subseteq\V,\\
|S'|=2k|S|\\
1\le |S|\le n^{\gamma}
\end{subarray}} 
e(S,S')\ge 4(k+1)|S|}\\
&\quad\le
\sum_{s=1}^{n^{\gamma}}
\binom{n}{s}\binom{n-s}{2ks}
\Bigg[\binom{2ks^2}{(2k+2)s}(\pdwa+(n-(2k+1)s)\ptrzy)^{(2k+2)s}\\
&\quad\quad+
\binom{\binom{s}{2}2ks+\binom{2ks}{2}s}{(k+1)s}\ptrzy^{(k+1)s}
\Bigg]
\\
&\quad\le 
\sum_{s=1}^{n^{\gamma}}
\Bigg[
\left(O(1)\left(\frac{n}{s}\right)^{2k+1}\left(\frac{s\ln n}{n}\right)^{2k+2}\right)^s\\
&\quad\quad+
\left(O(1)\left(\frac{n}{s}\right)^{2k+1}s^{k+1}\left(\frac{\ln n}{n^2}\right)^{k+1}\right)^s\Bigg]
\\
&\quad\le
\sum_{s=1}^{n^{\gamma}}
\Bigg[\left(
O(1)\frac{\ln^{2k+2}n}{n^{1-\gamma}}
\right)^s
+
\left(
O(1)\frac{\ln^{k+1} n}{n}
\right)^s\Bigg]
\\
&\quad=o(1).
\end{align*}
Similarly
\begin{align*}
&\Pra{
\exists_{\begin{subarray}{l} S\subseteq\V\\
1\le |S|\le n^{\gamma}
\end{subarray}} 
e(S)\ge 5|S|}\\
&\quad\le
\sum_{s=1}^{n^{\gamma}}
\binom{n}{s}
\Bigg[\binom{\binom{s}{2}}{2s}(\pdwa+(n-s)\ptrzy)^{2s}
+
\binom{\binom{s}{3}}{s}\ptrzy^{s}
\Bigg]\\
&\quad\le
\sum_{s=1}^{n^{\gamma}}
\left(
O(1)\frac{\ln^{2}n}{n^{1-\gamma}}
\right)^s
+
\left(
O(1)\frac{\ln n}{n^{1-\gamma}}
\right)^s
\\
&\quad=o(1).
\end{align*}
Therefore {\whp} for any set $S$ ($1\le |S|\le n^{\gamma}$) of vertices of degree at least $4k+14$ in $\Gsuma$ we have
$$
|N_3(S)|\ge 2k|S|.
$$
Otherwise for $S'=N_3(S)$ there would be $|S'|=|N_3(S)|\le 2ks$ and  $e(S,N_3(S))+2e(S)\ge (4k+14)|S|$, i.e. $e(S,N_3(S))\ge 4(k+1)|S|$ or $e(S)\ge 5|S|$.

\noindent(iv) The probability that in $\Gsuma$ there are two vertices of degree at most $C$ at distance at most $5$ is at most
\begin{align*}
&n^2\sum_{l=0}^{4}n^l(\pdwa+n\ptrzy)^{l+1}\\
&\cdot
\Bigg(\sum_{i=0}^{C}
\sum_{j=0}^{\lfloor i/2\rfloor}
\binom{\binom{n-1-l}{2}}{j}\ptrzy^{j}
\binom{n-1-l}{i-2j}\pdwa^{i-2j}\\
&\cdot
(1-\pdwa)^{n-1-i-l}
(1-\ptrzy)^{\binom{n-1-l}{2}-\binom{i+l}{2}}
\Bigg)^2\\
&=O(1)n(\ln n)^{2C+5}e^{-2n(\pdwa+\frac{n}{2}\ptrzy)}=o(1).
\end{align*}
\noindent
(v) Follows by Theorem~8.1 from \cite{KsiazkaBollobas}.
\end{proof}
\begin{lem}\label{LematMinDegreeGsuma}
Let $k$ be a positive integer and $\pdwa=\pdwa(n)\in(0,1)$, $\ptrzy=\ptrzy(n)\in (0,1)$ be such that
$$
\pdwa+\frac{n}{2}\ptrzy=\frac{\ln n + (k-1)\ln \ln n + c_n}{n}.
$$
\begin{itemize}
\item[(i)] If $k=1$ then  $$\Pra{\delta(\Gsuma)\ge 1}\to f(c_n),$$
where $f(\cdot)$ is defined by \eqref{RownanieFcn}.
\item[(ii)] If $c_n\to \infty$ then $$\Pra{\delta(\Gsuma)\ge k}\to 1.$$
\end{itemize}

\end{lem}
\begin{proof}
(i) Follows by Theorem~3.10 from~\cite{Magisterka}.

\noindent(ii) 
Let $X_t$ be the number of vertices of degree $0\le t\le k-1$ in $\Gsuma$. Then
\begin{align*}
\E X_t
&\le n\binom{n-1}{t}
\sum_{i=0}^{t}\binom{\binom{t}{2}}{i}\ptrzy^{i}\pdwa^{\max\{t-2i,0\}}
(1-\pdwa)^{n-t-1}(1-\ptrzy)^{\binom{n-1}{2}-\binom{t}{2}}\\
&=O(1)n(\ln n)^{t}\exp(-\ln n + (k-1)\ln \ln n + c_n)=o(1).
\end{align*}
Thus {\whp} $X_t=0$ for all $t\le k-1$.
\end{proof}

\begin{lem}\label{LematGnCkPM}
Let $k$ be a positive integer, $\liminf_{n\to\infty}\theta_n = \theta>\frac{1}{2}$, $\liminf_{n\to\infty}\theta_n' = \theta'>\frac{1}{2}$ and $\theta, \theta'=O(1)$.
Let $\rg(n)$ be a random graph such that 
for 
\begin{equation}\label{RownanieWarunekPdwaPtrzy}
\pdwa+\frac{n}{2}\ptrzy=\frac{\theta_n\ln n}{n}\text{ and }\pdwa=\frac{\theta'_n\ln n}{n}
\end{equation}
we have
\begin{equation}\label{RownanieCouplingGsumaGn}
\Gsuma\coup \rg(n)
\end{equation}
and in the probability space of the coupling {\whp} all vertices of degree at most $4k+13$ in $\Gsuma$ are at distance at least~$6$ in $\rg(n)$.
If $\Pra{\delta(\rg(n)\ge k)}$ is bounded away from zero by a constant then 
\begin{align}
&\Pra{\rg(n)\in \mathcal{C}_k|\delta(\rg(n))\ge k}\to 1,\label{RownanieGnCk}\\
&\Pra{\rg(2n)\in \mathcal{PM}|\delta(\rg(2n))\ge 1}\to 1 \label{RownanieGnPM}.
\end{align}
\end{lem}
In particular, if we substitute $\rg(n)=\Gsuma$ then by Lemmas~\ref{LematMinDegreeGsuma} and~\ref{LematGnCkPM} we obtain the following result.

\begin{lem}\label{LematGsumaCkPM}
Let $\pdwa$ and $\ptrzy$ fulfil \eqref{RownanieWarunekPdwaPtrzy}.
\begin{itemize}
\item[(i)] If 
$
\pdwa+\frac{n}{2}\ptrzy=(\ln n + c_n)/n
$ then
\begin{align*}
&\Pra{\Gsuma\in \mathcal{C}_1}\to f(c_n);\\
&\Pra{G_2(2n, \pdwa(2n))\cup G_3(2n, \ptrzy(2n))\in \mathcal{PM}}\to f(c_{2n}),
\end{align*}
where $f(\cdot)$ is defined by \eqref{RownanieFcn}.
\item[(ii)] If
$
\pdwa+\frac{n}{2}\ptrzy=(\ln n + (k-1)\ln \ln n +c_n)/n
$ and $c_n\to \infty$ then
\begin{align*}
&\Pra{\Gsuma\in \mathcal{C}_k}\to 1.
\end{align*}
\end{itemize}

\end{lem}

\begin{proof}[Proof of Lemma~\ref{LematGnCkPM}]
Denote by $\Gdelta$ a graph $\rg(n)$ under condition $\delta(\rg(n))\ge k$. Let ${\cal A}$ ba a graph property.
In what follows we will use the fact that if $\Pra{\delta(\rg(n)\ge k)}$ is bounded away from zero and $\rg(n)$ has ${\cal A}$ {\whp}, then also $\Gdelta$ has property ${\cal A}$ {\whp}. 

From now on we assume that $\Gsuma$ and $\rg(n)$ are defined on the same probability space existing by \eqref{RownanieCouplingGsumaGn}.
We call a vertex $v\in\V$ small if its degree in $\Gsuma$ is at most $4k+13$. Otherwise we call $v$ a large vertex. 
Let $S\subseteq\V$ and $|S|\le n^{\gamma}$. Denote by $S^+$ and $S^-$ the subset of large and small vertices of $S$, respectively. 
Then by Lemma~\ref{LematWlasnosciGsuma}(iii) {\whp} 
$$
|N_{\Gdelta}(S^+)|\ge |N_3(S^+)|\ge 2k|S^+|.
$$
Moreover in $\Gdelta$ all vertices in $S^-$ have degree at least $k$ and {\whp} are at distance at least $6$. Therefore {\whp} in $\Gdelta$ no two vertices in $S^-$ are connected by an edge or have a common neighbour and at most $|S^+|$ of them have neighbours in $N_3(S^+)\cup S^+$ (otherwise they would be connected by a path of length at most $5$). Thus {\whp} 
$$
\forall_{S\subseteq \V, 1\le |S|\le n^{\gamma}}
|N_{\Gdelta}(S)|\ge |N_{\Gdelta}(S^+)|+k\max\{|S^-|-|S^+|,0\}\ge k|S|.  
$$ 
If we combine this with Lemma~\ref{LematWlasnosciGsuma}(i) and (ii) we get that {\whp}
\begin{align}
&|N_{\Gdelta}(S)|\ge k, 
\hspace{3.3cm}
\text{ for all }S\subseteq \V,\ 1\le |S|\le \frac{n}{2}\label{RownanieNCk};\\
&|N_{\rg(n)_{\delta\ge 1}}(S)|\ge \min\{|S|,4n\ln \ln n/\ln n\}\ge \min\{|S|,\alpha(\Gdelta)\},\label{RownanieNPM}\\
&\nonumber \hspace{6.2cm}\text{ for all }S\subseteq \V,\ 1\le |S|\le \frac{2n}{3};\\
&|N_{\rg(n)_{\delta\ge 2}}(S)|\ge 2|S|, 
\hspace{2.8cm}
\text{ for all }S\subseteq \V,\ 1\le |S|\le \frac{n}{4}\label{RownanieNHC}.
\end{align}
Finally
\eqref{RownanieGnCk} follows immediately by \eqref{RownanieNCk}. Moreover if \eqref{RownanieNPM} is fulfilled then $\rg(n)_{\delta\ge 1}$ has a perfect matching (for the proof see for example Lemma~1 in~\cite{UniformMatchings}). Therefore \eqref{RownanieGnPM} follows. \eqref{RownanieNHC} will be used later to establish threshold function for a Hamilton cycle. 
\end{proof}

\begin{lem}\label{LematGnHamilton}
Let $k$ be a positive integer, $\theta_n$ and $\theta_n'$ be sequences such that $\liminf_{n\to\infty}\theta_n = \theta>\frac{1}{2}$, $\liminf_{n\to\infty}\theta_n' = \theta'>\frac{1}{2}$, $\theta, \theta'=O(1)$, and 
$$
\pdwa+\frac{n}{2}\ptrzy=\frac{\theta_n\ln n}{n},\ \pdwa=\frac{\theta'_n\ln n}{n}
\text{ and }
\phat_4=\frac{512}{n}.
$$
Let moreover $\rg(n)$ be a random graph such that: 
\begin{itemize}
\item[(i)]
$
\Gsuma\coup \rg(n);
$
\item[(ii)]
{\whp} $\delta(\rg(n))\ge 2$;
\item[(iii)]
in a probability space existing by (i) all vertices of degree at most $21$ in $\Gsuma$ are at distance at least~$6$ in $\rg(n)$.
\end{itemize}
Then
$$
\Pra{\rg(n)\cup\Gn{\phat_4}\in \mathcal{HC}}\to 1\label{RownanieHC}.
$$
\end{lem}
In particular by Lemma~\ref{LematWlasnosciGsuma} we get the following simple corollary of Lemma~\ref{LematGnHamilton}.
\begin{cor}\label{CorGnHamilton}
Let $\pdwa$ and $\ptrzy$ be as in Lemma~\ref{LematGnHamilton}. 
If moreover
$$
\pdwa+\frac{n}{2}\ptrzy=(\ln n + \ln \ln n +c_n)/n
\text{ where } c_n\to\infty$$ then
\begin{align*}
&\Pra{\Gsuma\in \mathcal{HC}}\to 1. 
\end{align*}
\end{cor}

\begin{proof}[Proof of Lemma~\ref{LematGnHamilton}]
We will use the heuristic of the proof of Theorem~8.9  \cite{KsiazkaBollobas}.
Let $t=8n/\ln n$ and
$
\phat_{4,0}=64\ln n/n^2.
$
Then 
$
t\phat_{4,0}=\phat_4.
$
For any graph $G$
let $l(G)$ be the length of the longest path in $G$ and $l(G)=n$ if $G$ has a Hamilton cycle. 
We say that $G$ has property $\mathcal{Q}$ if 
$$
\text{$G$ is connected\quad and }\quad|N_G(S)|\ge 2|S|\text{, for all }S\subseteq\V,|S|\le n/4.
$$
In the proof we assume that $\Gsuma$ and $\rg(n)$ are defined on the probability space of the coupling existing by (i).
Let $\rg_0=\rg(n)$  and
$$
\rg_i=\rg_{i-1}\cup \Gn{\phat_{4,0}},\quad \text{ for }1\le i\le t.
$$
By Lemma~\ref{LematGnCkPM}, assumptions (i)--(iii), and \eqref{RownanieNHC} {\whp} $\rg(n)$ has property $\mathcal{Q}$. Moreover by Lemma~\ref{LematWlasnosciGsuma}(v) {\whp} 
$$l(\rg_0)\ge \left(1-\frac{8\ln 2}{\ln n}\right)n .$$ 
It may be shown (see (8.7) in \cite{KsiazkaBollobas}) that
\begin{multline*}
\Pra{l(\rg_i)
=n-t+i-1|l(\rg_{i-1})=n-t+i-1 \text{ and } \rg_{i-1} \text{ has $\mathcal{Q}$}}\\
\le(1-\phat_{4,0})^{n^2/32}
\le n^{-2}.
\end{multline*}
Thus
$$
\Pra{l(\rg_t)=n}\ge 1-\frac{t}{n^2}-o(1)=1-o(1).
$$
Since  
$$
\rg_t\coupling \rg(n)\cup \Gn{t\phat_{4,0}}, 
$$
this finishes the proof.
\end{proof}

\section{Sharp thresholds}

\begin{proof}[Proof of Theorems~\ref{TwierdzenieSpojnosc}--\ref{TwierdzenieHamilton}]
First let $S_1=n(\ln n + c_n)$ and $\pdwa$ and $\ptrzy$ be given by \eqref{RownanieHatp}. Note that then $\pdwa>\frac{1+\varepsilon}{2}\ln n$ for some constant $\varepsilon>0$. Therefore by Theorem~\ref{LematCoupling} and Lemma~\ref{LematGsumaCkPM}(i)
$$
\lim_{n\to\infty}\Pra{\Gnmp\in\mathcal{C}_1}\ge \lim_{n\to\infty}\Pra{\Gsuma\in\mathcal{C}_1}=f(c_n)
$$
and by Lemma~\ref{LematStopnie}
$$
\lim_{n\to\infty}\Pra{\Gnmp\in\mathcal{C}_1}\le \lim_{n\to\infty}\Pra{\delta(\Gnmp)\ge 1}= f(c_n).
$$ 
This implies Theorem~\ref{TwierdzenieSpojnosc}(i).
Similarly 
Theorem~\ref{TwierdzenieSpojnosc}(ii) follows by Theorem~\ref{LematCoupling},  Lemma~\ref{LematGsumaCkPM}(ii), and
Lemma~\ref{LematStopnie2};
Theorem~\ref{TwierdzeniePM} by Theorem~\ref{LematCoupling}, Lemma~\ref{LematGsumaCkPM}(i), and
Lemma~\ref{LematStopnie};
Theorem~\ref{TwierdzenieHamilton} by Theorem~\ref{LematCoupling},  
Corollary~\ref{CorGnHamilton} and
Lemma~\ref{LematStopnie2}.
\end{proof}

In the following proofs we will assume that $c_n=O(\ln n)$. In the other cases the theorems follow by Lemma~\ref{LematStopnie} and  Theorem~\ref{LematCoupling} \eqref{RownanieCoupling1} combined with known results concerning $\Gn{\phat}$.
\begin{proof}[Proof of Theorem~\ref{TwierdzenieHamiltonGp}]
Let $c_n\to -\infty$, then 
by Lemma~\ref{LematStopniek} 
$$
\Pra{\Gnmpp\in\mathcal{HC}}\le \Pra{\delta(\Gnmpp)\ge 2}\to 0.
$$

Let now $c_n\to\infty$. We will consider two cases:

\noindent{\bf CASE 1:} $m= \Omega (n/\ln n)$.

\noindent Let
$$m'=m\left(1+\frac{700}{\ln n}\right)^{-1}
\quad\text{and}\quad
m''=\frac{700}{\ln n} m'.
$$
Then for $p$ given by \eqref{RownanieHamiltonGp}
$$
m'p(1-(1-p)^{n-1})=
\ln n +\ln \max\left\{1,\frac{ npe^{-np}\ln n}{1-e^{-np}}\right\} - 700 + o(1) + c_n
$$
and 
\begin{equation}\label{Rownanienmpkwadrat}
nm'p^2=
\begin{cases}
O(\ln n)&\text{for }np=O(1),\\
\frac{n}{m'}(m'p)^2=O(\ln^3n)&\text{for }np\to\infty.
\end{cases}
\end{equation}
Let $\omega\to \infty$.
Define
\begin{align*}
\pdwa&=
\begin{cases}\frac{m'p(1-(1-p)^{n-1})-3m'p\left(\frac{1-(1-2p)^n}{2np}-(1-p)^{n-1}\right)-\omega\sqrt{\frac{\ln n}{n}}}{n},
\\
\hspace{3cm}\text{ for }nm'p\left(\frac{1-(1-2p)^n}{2np}-(1-p)^{n-1}\right)\gg\sqrt{n\ln n};\\
\frac{m'p(1-(1-p)^{n-1})-\omega\sqrt{\frac{\ln n}{n}}}{n},
\quad\text{ otherwise };
\end{cases}\\
\ptrzy&=
\begin{cases}
\frac{m'p\left(\frac{1-(1-2p)^n}{2np}-(1-p)^{n-1}\right)-\omega\sqrt{\frac{\ln n}{n}}}{\frac{n^2}{6}},\\
\hspace{3cm}\text{ for }nm'p\left(\frac{1-(1-2p)^n}{2np}-(1-p)^{n-1}\right)\gg\sqrt{n\ln n};\\
0,\hspace{2.7cm}\text{ otherwise };
\end{cases}
\\
\phat_4&=\textstyle\frac{m''p\left(1-\frac{1-(1-2p)^n}{2np}\right)-\omega\sqrt{\frac{1}{n}}}{n}\ge \frac{512}{n}.
\end{align*}
Let now $\rg(n)=\Gnmprim{m'}$. Then
by Theorem~\ref{LematCoupling}
\begin{align}
\Gsuma&\coup\Gnmprim{m'}=\rg(n)\label{RownanieCoupmprim};\\
\Gn{\phat_4}&\coup\Gnmprim{m''}\nonumber.
\end{align} 
Moreover by Lemma~\ref{LematStopniek} {\whp} $\delta(\rg(n))\ge 2$.  
Therefore assumptions (i) and (ii) in Lemma~\ref{LematGnHamilton} are fulfilled.

We are left with proving that (iii) is fulfilled. Let $C$ be a positive integer. We will show that in the probability space of the coupling~\eqref{RownanieCoupmprim} {\whp}
any two vertices of degree at most $C$ in $\Gsuma$ are at distance at least $6$ in $\Gnmprim{m'}=\rg(n)$.
Therefore we need to study the couplings described in the proof of Theorem~\ref{LematCoupling}. Recall that in the proof of Theorem~\ref{LematCoupling} it was shown that, if in all definitions of  $S_1$, $S_3$, $M_2$ and $M_3$ we replace $m$ by $m'$, then
\begin{align*}
&\Gsuma\\
&=\G{2}{\Po{\binom{n}{2}\ln(1-\pdwa)^{-1}}}\cup\G{3}{\Po{\binom{n}{3}\ln (1-\ptrzy)^{-1}}}
\\
&\coupling^{(i)}
\G{2}{\Po{\frac{S_1-3S_3-5\omega'\sqrt{S_1}}{2}}}\cup\G{3}{\Po{S_3-2\omega'\sqrt{S_1}}}\\
&\coup^{(ii)} 
\G{2}{\frac{S_1-3S_3-4\omega'\sqrt{S_1}}{2}}\cup\G{3}{S_3-\omega'\sqrt{S_1}}\\
&\coup^{(iii)} 
\G{2}{M_2}
\cup 
\G{3}{M_3}=\bigcup_{1\le i \le m'}\left(\G{2}{\textstyle{\frac{Y_i-3Z_i}{2}}}\cup \G{3}{Z_i}\right)\\
&\coupling^{(iv)}
\Gnmprim{m'}
\end{align*}
for some $\omega'=\Theta(\omega)$.

In the probability space of the above couplings, we will find a relation between the degree of a vertex $v\in\V$  in $\Gsuma$ and the number of those features $w_i\in W(v)$ that are chosen by at least two vertices (i.e. that contribute to at least one edge in $\Gnmprim{m'}$). For any $v\in\V$ in $\Gnmprim{m'}$ let 
$$W'(v)=\{w_i\in W(v): |V'_i|\ge 2\},$$
where $V'_i$ is defined in \eqref{RownanieVprim}.
Therefore we will find the relation between the degree of a vertex $v$ in $\Gsuma$ and $|W'(v)|$ in $\Gnmprim{m'}$. 

To this end we will need an insight into the construction of the intermediary couplings (i)--(iii). Recall that in the couplings, the construction of $\G{2}{M_2}\cup \G{3}{M_3}$ proceeds in $m'$ rounds. In the $i$-th round $\G{2}{\textstyle{\frac{Y_i-3Z_i}{2}}}\cup \G{3}{Z_i}$ is constructed in $\frac{Y_i-3Z_i}{2}+Z_i$ draws. In each draw a hyperedge of size $2$ or $3$ in $\mathbb{H}_{*2}(\frac{Y_i-3Z_i}{2})\cup \mathbb{H}_{*3}(Z_i)$  is chosen.  

Denote by $\mathcal{A}_v$ the event that while constructing $\G{2}{M_2}
\cup 
\G{3}{M_3}$ there are less than  $|W'(v)|-1$ draws in which  a hyperedge containing $v$ is chosen. 
Given $v\in \V$ and $1\le i\le m'$, denote by $\mathcal{A}_{v,i,1}$ the event that $v$ is an isolated vertex in $\G{2}{(Y_i-3Z_i)/2}\cup\G{3}{Z_i}$ and $\mathcal{A}_{v,i,2}$ the  event that $v\in V_i'$. Then 
$$\mathcal{A}_v\subseteq 
\bigcup_{i\neq j}
(\mathcal{A}_{v,i,1}\cap\mathcal{A}_{v,i,2})
\cap
(\mathcal{A}_{v,j,1}\cap\mathcal{A}_{v,j,2}).$$ 
For $i\neq j$, $(\mathcal{A}_{v,i,1}\cap\mathcal{A}_{v,i,2})$ and $(\mathcal{A}_{v,j,1}\cap\mathcal{A}_{v,j,2})$ are independent and $$\Pra{\mathcal{A}_{v,i,1}\cup\mathcal{A}_{v,i,2}}=1,$$ thus
\begin{align*}
\Pra{\mathcal{A}_v}
& \le 
\binom{m'}{2}\left(\Pra{\mathcal{A}_{v,1,1}\cap\mathcal{A}_{v,1,2}}\right)^2\\
& = 
\binom{m'}{2}\left(\Pra{\mathcal{A}_{v,1,1}}+\Pra{\mathcal{A}_{v,1,2}}-1\right)^2
\\
&\le\binom{m'}{2} 
\Bigg(
(1-p)^{n}
+np(1-p)^{n-1}
\\
&\hspace{2cm}+\sum_{y=2}^{n-1}
\binom{n}{y}p^y(1-p)^{n-y}
\left(1-\frac{1}{n}\right)^y
+p^{n}\cdot 0\\
&\hspace{2cm}
+\sum_{y=2}^{n}
\binom{n}{y}p^y(1-p)^{n-y}
\frac{y}{n}
-1
\Bigg)^2\\
&=\binom{m'}{2} 
\Bigg(
\sum_{y=0}^{n}
\binom{n}{y}p^y(1-p)^{n-y}
\left(1-\frac{1}{n}\right)^y
-p^{n}\\
&\hspace{2cm}+\sum_{y=1}^{n}
\binom{n}{y}p^y(1-p)^{n-y}
\frac{y}{n}
-1
\Bigg)^2\\
&\le\binom{m'}{2}
\left(
\left(1-\frac{p}{n}\right)^n
+p-1
\right)^2
\le (m')^2p^4
\end{align*}
Therefore by \eqref{Rownanienmpkwadrat}
\begin{equation}\label{RownanieAv}
\Pra{\bigcup_{v\in\V}\mathcal{A}_v}\le n(m')^2p^4=O\left(\frac{\ln^6n}{n}\right)=o(1).
\end{equation}
Thus in the probability space of the couplings, {\whp}  for all $v\in\V$   in the construction of $\G{2}{M_2}\cup\G{3}{M_3}$ the number of draws  in which a hyperedge containing $v$ is chosen is at least $|W'(v)|-1$. 

Now we will find the relation between $|W'(v)|-1$ and the number of draws in the construction of 
\begin{multline*}
\Gsuma\\
=\G{2}{\Po{\binom{n}{2}\ln(1-\pdwa)^{-1}}}\cup\G{3}{\Po{\binom{n}{3}\ln (1-\ptrzy)^{-1}}}.
\end{multline*}   
Note that in the couplings (i)--(iii) in order to construct a graph $\G{2}{M_2}
\cup 
\G{3}{M_3}$ from $\Gsuma$  one make some additional draws of hyperedges in auxiliary hypergraphs $\mathbb{H}_{*2}(n,\cdot)\cup \mathbb{H}_{*3}(n,\cdot)$. By the sharp concentration of the Poisson distribution and \eqref{RownanieYikoncentracja} {\whp} the number of additional draws is at most $K\omega\sqrt{n\ln n}$ for some constant $K$. Under condition that the number of additional draws is  at most $K\omega\sqrt{n\ln n}$,  probability that a hyperedge  containing $v$ is chosen in at least $3$ of the additional draws is at most.
$$
\binom{K\omega\sqrt{n\ln n}}{3}\left(\frac{2}{n}\right)^3=o(n^{-1}).
$$
Therefore by the union bound and \eqref{RownanieAv} {\whp} for all $v\in\V$ in the construction of $\Gsuma$  the number of draws in which a hyperedge containing $v$ is chosen is at least $|W'(v)|-3$. 

Now we will show the relation between the number of draws and the number of edges incident to $v$ in $\Gsuma$.
Note that the number of draws made in the construction of $\Gsuma$ is a random variable with the expected value at most $S_1=O(n\ln n)$. Therefore by Markov's inequality {\whp} the number of draws made in the construction of $\Gsuma$ is at most $n\ln^2 n$. In addition the probability that a pair $\{v,v'\}$ is contained in a hyperedge chosen in a given draw is at most $6/n(n-1)$. Therefore the probability that there exists an edge in $\Gsuma$ drawn at least $3$ times during the construction is at most
$$
\binom{n}{2}\binom{n\ln^2 n}{3}\left(\frac{6}{n(n-1)}\right)^3+o(1)=o(1)
$$   
Concluding, we have that {\whp} for all $v\in\V$ the number of edges incident to $v$ in $\Gsuma$ is at least half of the number of drawn hyperedges containing $v$, which is {\whp} at least $\frac{|W'(v)|-3}{2}$. Thus in the probability space of the couplings {\whp} for all $v\in\V$  of degree at most $C$ in $\Gsuma$ we have $|W'(v)|\le 2C+6$.

Finally,
probability that there are two vertices $v,v'$ such that $|W'(v)|\le 2C+6$ and $|W'(v')|\le 2C+6$ connected by a path of length at most $5$ in $\Gnmprim{m'}$ is at most
\begin{align*}
&n^2
\sum_{t=1}^{5}
(m')^tn^{t-1}p^{2t}\\
&\cdot\left(
\sum_{l=0}^{2C+6}
\binom{m'-t}{l}
(p-p(1-p)^{n-1})^l\left[(1-p)+p(1-p)^{n-1}\right]^{m'-O(1)}
\right)^2\\
&=O(1)n(\ln n)^{O(1)}
\exp\left(-2m'p(1-(1-p)^{n-1})\right)\\
&=o(1).
\end{align*}
This shows that assumption (iii) of Lemma~\ref{LematGnHamilton} is fulfilled. 

Therefore by Lemma~\ref{LematGnHamilton}
{\whp} 
$$
\rg(n)\cup\Gn{\frac{512}{n}}\in \mathcal{HC}.
$$
Thus in the case $m = \Omega(n/\ln n)$ the theorem follows by \eqref{RownanieCoupmprim} and a straight forward coupling of random intersection graphs. 
$$
\rg(n)\cup\Gn{\frac{512}{n}}
\coup
\Gnmprim{m'}\cup\Gnmprim{m''}
\coupling
\Gnmpp
$$ 

\noindent{\bf CASE 2:} $m = o (n/\ln n)$.\\
Note that in this case \eqref{RownanieHamiltonGp} and $c_n=O(\ln n)$ imply that
$$
p=\frac{\ln n + c'_n}{m}
$$
for some $c'_n\to\infty$. Let $m'$ be such that $m$ divides $m'$ and $m'\sim n/\ln n$.
Take an instance of $\mathcal{G}(n,m',p')$ with the set of features $\W'$ of size $m'$ and  
$$
p'=\frac{\ln n + c'_n}{m'}.
$$ 
Divide $\W'$ into $m$ groups of $m'/m$ features. Denote by $A_i$ the set of vertices which have chosen features from the $i$--th, $1\le i \le m$, group in $\mathcal{G}(n,m',p')$. $|A_i|$ has the binomial distribution $\Bin{n}{p''}$ for some $p''\le \frac{m'}{m}p'=p$. Now take an instance of $\mathcal{G}(n,m',p')$ and construct the sets $V_i$, $1\le i \le m$, in $\Gnmpp$ by taking $A_i$ and additionally adding to $V_i$ each vertex from $\V\setminus A_i$ independently with probability $(p-p'')/(1-p'')$. This coupling implies
$$
\mathcal{G}(n,m',p')\coupling \Gnmpp.
$$
From the considerations concerning the first case we have that {\whp}
$\mathcal{G}(n,m',p')\in\mathcal{HC}$. Therefore also $\Gnmpp\in\mathcal{HC}$ {\whp} . 
\end{proof}

\begin{proof}[Proof of Theorem~\ref{TwierdzenieSpojnosckGp}]
The technique of the proof is analogous to this of the proof of Theorem~\ref{TwierdzenieHamiltonGp}. We consider two cases $m=\Omega(n/\ln n)$ and $m=o(n/\ln n)$. In the first case the proof relies on Lemma~\ref{LematStopniek} and Lemma~\ref{LematGnCkPM}. Moreover we have to use the fact shown in the proof of Theorem~\ref{TwierdzenieHamiltonGp}, that {\whp} in the coupling constructed in the proof of Theorem~\ref{LematCoupling} {\whp} the vertices of degree bounded by a constant in $\Gsuma$ are at distance at least $6$ in $\Gnmpp$. The proof of the case $m=o(n/\ln n)$ is the same as in the proof of Theorem~\ref{TwierdzenieHamiltonGp}.   
\end{proof}

\section*{Acknowledgements}
Katarzyna Rybarczyk acknowledges a support of the National Science Center grant DEC--2011/01/B/ST1/03943.

\appendix

\section{Proofs of Lemma~\ref{LematStopnie2} and \ref{LematStopniek}}

\begin{proof}[Proof of Lemma~\ref{LematStopnie2}]
In the proof we use notation introduced in \eqref{RownanieTp}. Moreover let $a_n=S_{1,2}/S_1$.
\medskip

\noindent(i)
We restrict attention to the case $c_n=o(\ln n)$. In the latter case the result follows by Lemma~\ref{LematStopnie}.

First consider the case  $a_n=S_{1,2}/S_1\gg 1/\ln n$. Then $S_{1,2}\to \infty$. 
Take any probability space on which we define the coupon collector process on $\V$ and $\Gn{\phat_{2+}}$, in such a manner that they are independent.
Let $X_{+}$ be a random variable counting vertices which have not been collected during the coupon collector process in $T_{2+}$ draws and have degree at most $k-1$ in $\Gn{\phat_{2+}}$. If $$S_1=n(\ln n + (k-1)\ln (a_n\ln n) + c_n)$$ then for $\omega$ tending to infinity slowly enough
\begin{align*}
\frac{1}{n}T_{2+} + (n-1)\phat_{2+}
&=\frac{1}{n}\left(
S_1+O\left(\omega\sqrt{S_1}+\frac{S_1}{\omega\ln n}\right)\right)\\
&=\ln n + (k-1)\ln(a_n\ln n) + c_n + o(1)
\end{align*}
and
\begin{align*}
\phat_{2+}
\sim \frac{a_n \ln n}{n}.
\end{align*}
Therefore
\begin{align*}
\E X_+ 
=& n\left(1-\frac{1}{n}\right)^{T_{2+}}
\sum_{i=0}^{k-1}\binom{n-1}{i}\phat_{2+}^{i}(1-\phat_{2+})^{n-i-1}
\\
\nonumber \sim&  n\frac{1}{(k-1)!}(a_n\ln n)^{k-1} \exp(-\ln n - (k-1)\ln(a_n\ln n) - c_n)\\
\sim& \frac{e^{-c_n}}{(k-1)!}\\
\intertext{and}
\E X_+(X_+-1)
\sim& n^2\left(1-\frac{2}{n}\right)^{T_{2+}}(1-\phat_{2+})
\left(\sum_{i=0}^{k-1}\binom{n-2}{i}\phat_{2+}^{i}(1-\phat_{2+})^{n-i-2}\right)^2
\\
\nonumber&\ +n^2\left(1-\frac{2}{n}\right)^{T_{2+}}
\phat_{2+}\left(\sum_{i=0}^{k-2}\binom{n-2}{i}\phat_{2+}^{i}(1-\phat_{2+})^{n-i-2}\right)^2\\
\nonumber \sim& \left(\frac{e^{-c_n}}{(k-1)!}\right)^2.
\end{align*}
Thus by the second moment method {\whp} $X_+>0$ as $c_n\to -\infty$. Note that if there is a vertex which is isolated in $\bigcup_{t=3}^{k+1}\Gnm{t}{\overline{p}}$ and has degree at most $k-1$ in $\Gnm{2}{\overline{p}}$, then $\delta(\Gnmp)\le k-1$.  By Facts~\ref{FaktStopnieCoupling} and~\ref{FaktCollectorPlus} there is a probability space such that {\whp} $\Gnm{2}{\overline{p}}\subseteq \Gn{\phat_{2+}}$, the number of isolated vertices in $\bigcup_{t=3}^{k+1}\Gnm{t}{\overline{p}}$ is at least the number of non--collected coupons after $T_{2+}$ draws, and the coupon collector process and $\Gn{\phat_{2+}}$ are independent. Thus $X_+>0$ implies that {\whp} $\delta(\Gnmp)\le k-1$.

Now let $S_{1,2}/S_1= O ((\ln n)^{-1})$. Then $
S_1=n(\ln n + c_n + O(1)).
$
Therefore for the result follows by Lemma~\ref{LematStopnie}.

\medskip

\noindent(ii)  
In the proof we restrict our attention to the case $c_n=O(\ln n)$. In the latter case the result follows after combining \eqref{RownanieCoupling1} with known results on $\Gn{\phat}$.

Let  $a_n=S_{1,2}/S_1\gg 1/\ln n$ (thus $S_{1,2}\to\infty$ and $a_n\ln n\to\infty$). Consider any probability space on which we define the coupon collector process on $\V$ and $\Gn{\phat_{2-}}$, in such a manner that they are independent.
Let $X_{-}$ be a random variable defined on this probability space and counting vertices which have not been collected during the coupon collector process with $T_-$ draws and have degree at most $k-1$ in $\Gn{\phat_{2-}}$. 
If $S_1-\sum_{t=3}^{k}S_{1,t}=n(\ln n + (k-1)\ln (a_n\ln n) + c_n)$ then for $\omega$ tending to infinity slowly enough
\begin{align*}
\frac{1}{n}T_- + (n-1)\phat_{2-}
&\ge\frac{1}{n}\left(
S_1-\sum_{t=3}^{k-1}S_{1,t}+O\left(\omega\sqrt{S_1}+S_{1}^2n^{-2}\right)\right)\\
&=\ln n + (k-1)\ln(a_n\ln n) + c_n + o(1).
\end{align*}
Thus
\begin{align*}
\E X_- 
&= n\left(1-\frac{1}{n}\right)^{T_{-}}
\left(\sum_{i=0}^{k-1}\binom{n-1}{i}\phat_{2-}^{i}(1-\phat_{2-})^{k-1-i}\right)
(1-\phat_{2-})^{n-k}\\
\nonumber
&= O(1)  n\frac{1}{(k-1)!}(a_n\ln n)^{k-1} \exp(-\ln n - (k-1)\ln(a_n\ln n) - c_n)\\
&= O(1) \frac{e^{-c_n}}{(k-1)!}
\end{align*}
Therefore {\whp} $X_-=0$, i.e. {\whp} each vertex is collected in $T_-$ draws or has degree at least $k$ in  $\Gn{\phat_{2-}}$.  Note that if each vertex is non-isolated in $\Gnm{k+1}{\overline{p}}$  or has degree at least $k$ in $\Gnm{2}{\overline{p}}$, then $\delta(\Gnmp)\ge k$. Therefore Facts~\ref{FaktStopnieCoupling} and~\ref{FaktCollectorMinus} imply that {\whp} $\delta(\Gnmp)\ge k$.

Now let $S_{1,2}/S_1= O ((\ln n)^{-1})$, i.e. $S_{1,2}=O(n)$. Then $$
S_1-\sum_{t=2}^{k-1}S_{1,t}=n(\ln n + c_n + O(1))
$$
and
$$
T_-
=n(\ln n + c_n + O(1)). 
$$
Thus by Fact~\ref{FaktCollectorMinus} {\whp} there are no isolated vertices in $\Gnm{k+1}{\overline{p}}$, i.e. {\whp} $\delta(\Gnmp)\ge k$.
\end{proof}

\begin{proof}[Proof of Lemma~\ref{LematStopniek}]
We use notation from \eqref{RownanieTp}. Moreover let
$$
a_n=
(np)^{k-1} \left( \left(\frac{e^{-np}\ln n}{1-e^{-np}}\right)^{k-1}+\frac{e^{-np}\ln n}{1-e^{-np}}\right).
$$
Note that 
\begin{multline}\label{RownanieS1tS1}
S_{1,t}\sim m\frac{(np)^{t}}{(t-1)!}(1-p)^{n}
\sim ((t-1)!)^{-1}\frac{(np)^{t-1}(1-p)^{n}}{1-(1-p)^{n-1}}S_1,\\
\quad\text{for }t=2,\ldots,k, 
\end{multline}
where $S_1$ and $S_{1,t}$ are defined by \eqref{RownanieS}. Therefore if $np^2=o(1)$ then
$$
\left(\frac{S_{1,2}}{S_1}\ln n\right)^{k-1} + (k-1)!\frac{S_{1,k}}{S_1}\ln n
\sim 
a_n.
$$
We will divide the proof into cases depending on the value of $S_{1,t}/S_1$, for $t=2,k$. 
\smallskip

\noindent {\bf CASE 1.} $S_{1,2}/S_1=O(1/\ln n)$ and $S_{1,k}/S_1=O(1/\ln n)$. 
\smallskip

Note that by \eqref{RownanieS1tS1} this case includes $np^2=\Omega(1)$. Moreover in this case $a_n=O(1)$ and $\frac{S_{1,t}}{S_1}\ln n=O(1)$ for all $2\le t\le k$. Therefore the lemma follows by Lemma~\ref{LematStopnie2}. 
\smallskip

\noindent {\bf CASE 2.} $S_{1,2}/S_1\gg 1/\ln n$ or $S_{1,k}/S_1\gg 1/\ln n$. 
\smallskip

Let $c_n\to -\infty$. If $S_{1,k}=O(S_{1,2})$ then the result follows by Lemma~\ref{LematStopnie2}. In the case  
$S_{1,k}\gg S_{1,2}$ we may restrict attention to $c_n=o(\ln n)$ (for explanation see the proof of  Lemma~\ref{LematStopnie2}(i)). 
Since $S_1=\Theta(n\ln n)$ and $S_{1,k}\gg S_1/\ln n$ we have $S_{1,k}\to\infty$. Thus we may apply the second part of Fact~\ref{FaktStopnieCoupling}.

Take any probability space on which we define independent coupon collector process on $\V$ and $G_{k}(n,\phat_{k+})$.
Let $X_{+}$ be a random variable equal to the number of vertices which have not been collected during the coupon collector process in $T_{k+}$ draws and have degree at most $k-1$ in $G_k(n,\phat_{k+})$. For $\omega$ tending to infinity slowly enough

\begin{align*}
\frac{1}{n}T_{k+} + \binom{n-1}{k-1}\phat_{k+}
&=\frac{1}{n}\left(
S_1+O\left(\omega\sqrt{S_1}+S_{1}^2n^{-2}+\frac{S_1}{\omega\ln n}\right)\right)\\
&=\ln n + \ln a_n  + c_n + o(1),
\end{align*}
thus
\begin{align*}
\E X_+ 
=& n\left(1-\frac{1}{n}\right)^{T_{k+}}\\
&\cdot
\left(
(1-\phat_{k+})^{\binom{n-1}{k-1}}
+
\binom{n-1}{k-1}\phat_{k+}(1-\phat_{k+})^{\binom{n-1}{k-1}-1}
\right)
\\
\sim&  n\frac{1}{(k-1)!}a_n \exp(-\ln n - \ln a_n - c_n)\sim \frac{e^{-c_n}}{(k-1)!}\\
\intertext{and}
\E X_+(X_+-1)
\sim& 
n^2
\left(1-\frac{2}{n}\right)^{T_{k+}}\\
&\cdot\Big[
\left(
1
+
\binom{n-2}{k-1}\phat_{k+}(1-\phat_{k+})^{-1}
\right)^2
(1-\phat_{k+})^{2\binom{n-1}{k-1}-\binom{n-2}{k-2}}\\
&
+
\binom{n-2}{k-2}\phat_{k+}
(1-\phat_{k+})^{2\binom{n-1}{k-1}-\binom{n-2}{k-2}-1}
\Big]\\
\sim& \left(\frac{e^{-c_n}}{(k-1)!}\right)^2.
\end{align*}
Thus {\whp} $X_+>0$ as $c_n\to -\infty$ (i.e. {\whp} there exists a vertex, which has degree at most $k-1$ in $G_k(n,\phat_{k+})$ or has not been collected in $T_{k+}$ draws). Thus by Facts~\ref{FaktStopnieCoupling} and~\ref{FaktCollectorPlus} {\whp} $\delta(\Gnmpp)\le k-1$.

\medskip

Now assume that $c_n\to \infty$. As it is explained in the proof of Lemma~\ref{LematStopnie2} we may restrict considerations to the case $c_n=O(\ln n)$. Recall that we assume that $S_{1,2} \gg S_1/\ln n$ or $S_{1,t}\gg S_1/\ln n$, i.e. $a_n\to\infty$.  
Note that by definition
$$
\hat{q}_{t}\le \frac{S_{1,t}}{t\binom{n}{t}}= O\left(mp^{t}e^{-np}\right).
$$
Take any probability space on which we may define independent coupon collector process with the set of coupons $\V$ and $\bigcup_{t=2}^{k}G_{t}(n,\hat{q}_{t})$. 
Let $X_{-}$ be a random variable equal to the number of vertices which have not been collected during the coupon collector process in $T_-$ draws and have degree at most $k-1$ in $\bigcup_{t=2}^{k}G_t(n,\hat{q}_{t})$. 
Note that if $v$ has degree at most $k-1$ in $\bigcup_{t=2}^{k}G_t(n,\hat{q}_{t})$, then for some $0\le k_0\le k-1$ and some sequence of integers $r_2,\ldots,r_t$ such that $\sum_{t=2}^{k}(t-1)r_t=k_0$  
\begin{itemize}
\item[(i)] there is a set $\V'\subseteq \V\setminus\{v\}$ of $k_0$ vertices such that for each $t$ there are $r_t$ hyperedges in $H_t(n,\hat{q}_{t})$ (generating $G_t(n,\hat{q}_{t})$) containing $v$ and contained in $\V'\cup\{v\}$.
\item[(ii)] 
and all hyperedges in $\bigcup_{t=2}^{k}H_t(n,\hat{q}_{t})$ containing $v$ are subsets of $\V'\cup\{v\}$.
\end{itemize}
Let $r=\sum_{t=2}^{k}r_t$, then for $c_n=O(\ln n)$ event (i) occurs with probability at most
\begin{align*}
& \binom{n-1}{k_0}\prod_{t=2}^{k}\binom{k_0-1}{t-1}^{r_t}\left(O(1)(mp^{t})e^{-np}\right)^{r_t}\\
&=O(1)n^{k_0}m^{r}p^{k_0+r}\left(e^{-np}\right)^{r}\\
&=O(1)(np)^{k_0}(mpe^{-np})^{r}\\
&=O(1)(np)^{k_0}\left(\frac{\ln n}{(1-e^{-np})}e^{-np}\right)^r\\
&=O(a_n).
\end{align*}
There are $O(1)$ sequences of integers $r_2,\ldots,r_t$ such that $\sum_{t=2}^{k}(t-1)r_t\le k-1$, thus the probability that there exists such sequence $r_2,\ldots,r_t$ that (i) and (ii) occur is at most 
$$
O(a_n)\prod_{t=2}^{k}(1-\hat{q}_{t})^{\binom{n-1}{t-1}-\binom{k-1}{t-1}}.
$$
Moreover
\begin{align*}
\frac{1}{n}T_- + \sum_{t=2}^{k}\binom{n-1}{t-1}\hat{q}_{t}
&\ge\frac{1}{n}\left(
S_1-O\left(\omega\sqrt{S_1}+S_{1}^2n^{-2}\right)\right)\\
&=\ln n + \ln a_n + c_n + o(1).
\end{align*}
Therefore finally
\begin{align*}
\E X_- 
&= n\left(1-\frac{1}{n}\right)^{T_{-}}
O(a_n)\prod_{t=2}^{k}(1-\hat{q}_{t})^{\binom{n-1}{t-1}-\binom{k-1}{t-1}}=o(1).
\end{align*}
Concluding, {\whp} $X_-=0$. By Facts~\ref{FaktStopnieCoupling} and~\ref{FaktCollectorMinus} {\whp} $\delta(\Gnmpp)\ge k$.
\end{proof}


\begin{thebibliography}{9}

\bibitem{Gppdistance}
\textsc{Barbour, A.D.} and \textsc{Reinert, G.} (2011) The shortest distance in random multi-type intersection graphs.
\textit{Random Structures \& Algorithms}
\textbf{39} 179--209.


\bibitem{RIGClustering2}
\textsc{Bloznelis, M.} (2013+). Degree and clustering coefficient in sparse random intersection graphs. 
\textit{Ann. Appl. Probab.} \textbf{23} 1254-1289.


\bibitem{UniformMatchings}
\textsc{Bloznelis, M.} and \textsc{{\L}uczak, T.} (2013).
Perfect matchings in random intersection graphs. \textit{Acta Math. Hungar}
\textbf{138} 15--33.


\bibitem{GppDegrees}
\textsc{Bloznelis, M.} and \textsc{Damarackas, J.} (2012+).
Degree distribution of an inhomogeneous random intersection graph. \textit{submitted}
.


\bibitem{WSNphase2}
\textsc{Bloznelis, M., Jaworski, J.}, and \textsc{Rybarczyk, K.} (2009).
Component evolution in a secure wireless sensor network. \textit{Networks}
\textbf{53} 19--26.


\bibitem{KsiazkaBollobas}
\textsc{Bollob\'{a}s, B. } (1985). \textit{Random Graphs},
Academic Press.

\bibitem{GppPhaseTransition}
\textsc{Bradonji\'{c}, M. , Hagbergy, A., Hengartnerz, N. W., Lemons, N.} and \textsc{Percus, A. G.} (2010).
Component Evolution in General Random Intersection Graphs.
\textit{Algorithms and Models for
the Web-Graph}, LNCS 6516, Springer, 36--49.



\bibitem{GpEpidemics}
\textsc{Brittom, T., Deijfen, M., Lager\r{a}s, A. N.,} and \textsc{Lindholm, M.} (2008).
Epidemics on random graphs with tunable clustering. \textit{J. Appl. Probab}
\textbf{45} 743--756.



\bibitem{RIGTunableDegree}
\textsc{Deijfen, M.} and \textsc{Kets, W.} (2009).
Random Intersection Graphs with Tunable Degree Distribution and Clustering. \textit{Probab. Engrg. Inform. Sci.}
\textbf{23} 661--674.


\bibitem{GpEquivalence}
\textsc{Fill, J. A., Scheinerman, E. R.,} and \textsc{Singer--Cohen K. B.} (2000).
Random Intersection Graphs when $m=\omega(n)$: An Equivalence Theorem Relating the Evolution of the {G}(n, m, p) and {G}(n, p) Models. \textit{Random Structures Algorithms}
\textbf{16} 156--176.


\bibitem{RIGGodehardt1}
\textsc{Godehardt, E.} and \textsc{Jaworski, J.} (2003).
Two Models of Random Intersection Graphs for Classification. 
\textit{Stud. Classiﬁcation Data Anal. Knowledge Organ.}
\textbf{22}, Springer, 67--81.

\bibitem{KsiazkaJLR}
\textsc{Janson, S., {\L}uczak, T.,} and \textsc{Ruci\'{n}ski, A.} (2001). \textit{Random Graphs},
Wiley.


\bibitem{GpSubgraph}
\textsc{Karo\'{n}ski, M., Scheinerman, E. R.,} and  \textsc{Singer-Cohen, K.B.} (1999).
On random intersection graphs: The subgraph problem. \textit{Combin. Probab. Comput.}
\textbf{8} 131--159.


\bibitem{SpirakisNiezalezne}
\textsc{Nikoletseas, S., Raptopoulos, C.,} and \textsc{Spirakis, P.} (2008).
Large independent sets in general random intersection graphs. \textit{Theoret. Comput. Sci.}
\textbf{406} 215–-224.


\bibitem{GpCoupling}
\textsc{Rybarczyk, K.} (2011).
Sharp Threshold Functions for Random Intersection Graphs via a Coupling Method. \textit{Electron. J. Combin}
\textbf{18} P36.


\bibitem{GpEquivalence2}
\textsc{Rybarczyk, K.} (2011).
Equivalence of the random intersection graph and {G}(n,p). \textit{Random Structures Algorithms}
\textbf{38} 205--234.


\bibitem{Magisterka}
\textsc{Rybarczyk, K.} (2005).
\textit{O pewnych zastosowaniach hipergraf\'ow losowych ({O}n some applications of random intersection graphs)}
Master Thesis, Faculty of Mathematics and Computer Science, Adam Mickiewicz University, Pozna\'{n}.

\bibitem{GpSubgraphPoisson}
\textsc{Rybarczyk, K.} and \textsc{Stark, D.} (2010).
Poisson approximation of the number of cliques in the random intersection graph. \textit{J. Appl. Probab}
\textbf{47} 826 -- 840.


\bibitem{SingerPhD}
\textsc{Singer, K. B.} (1995).
\textit{Random intersection graphs.} 
PhD Thesis, Department of Mathematical Sciences, The Johns Hopkins University.

\end{thebibliography}
\end{document}